\documentclass[twocolumn,10pt]{IEEEtran}
\usepackage{setspace}
\usepackage{cite}
\usepackage{graphicx}
\usepackage{subfigure}
\usepackage{amsmath}
\usepackage{amssymb}
\usepackage{amsthm}
\usepackage{amsmath}
\usepackage{algorithm,algorithmic}
\usepackage{booktabs}
\usepackage{multirow}
\usepackage{color}

\newcommand{\bdsb}[1]{\boldsymbol{#1}}

\DeclareMathOperator*{\argmin}{arg\,min}

\newtheorem{lem}{Lemma}

\newtheorem{thm}{Theorem}
\newtheorem{rmk}{Remark}
\newtheorem{proposition}{Proposition}

\usepackage{color}

\graphicspath{{./figure/}}

\title{A Framework for Phasor Measurement Placement in Hybrid State Estimation via Gauss-Newton}

\author{Xiao Li, Anna Scaglione and Tsung-Hui Chang\thanks{This work was supported by DOE under the TCIPG project. Part of this work was presented at the ICASSP conference \cite{li2013optimal}. }
}

\begin{document}

\maketitle

\begin{abstract}
In this paper, we study the placement of Phasor Measurement Units (PMU) for enhancing hybrid state estimation via the traditional Gauss-Newton method, which uses measurements from both PMU devices and Supervisory Control and Data Acquisition (SCADA) systems.
To compare the impact of PMU placements, we introduce a useful metric which accounts for three important requirements in power system state estimation: {\it convergence}, {\it observability} and {\it performance} (COP). Our COP metric can be used to evaluate the estimation performance and numerical stability of the state estimator, which is later used to optimize the PMU locations. In particular, we cast the optimal placement problem in a unified formulation as a semi-definite program (SDP) with integer variables and constraints that guarantee observability in case of measurements loss. Last but not least, we propose a relaxation scheme of the original integer-constrained SDP with randomization techniques, which closely approximates the optimum deployment. Simulations of the IEEE-30 and 118 systems corroborate our analysis, showing that the proposed scheme improves the convergence of the state estimator, while maintaining optimal asymptotic performance.
\end{abstract}


\vspace{-0.4cm}
\section{Introduction}
Power system state estimation (PSSE), using non-linear power measurements from the Supervisory Control and Data Acquisition (SCADA) systems, is plagued by ambiguities and convergence issues. Today, the more advanced Phasor Measurement Units (PMU) deployed in Wide-Area Measurement Systems (WAMS), provide synchronized voltage and current phasor readings at each instrumented bus, by leveraging the GPS timing information. PMUs data benefit greatly state estimation \cite{chow2011guidelines} because, if one were to use only PMUs, the state can be obtained as a simple linear least squares solution, in one shot  \cite{yang2011transition}. However, the estimation error can be quite high, and the system can loose even observability, due to the limited deployment of PMUs. For this reason researchers have proposed hybrid state estimation schemes \cite{phadke1986state}, integrating both PMU and SCADA data. Some of these methods incorporate the PMU measurements into the iterative state estimation updates \cite{meliopoulos2010supercalibrator,qin2007hybrid,nuqui2007hybrid}, while others use PMU data to refine the estimates obtained from SCADA data \cite{chakrabarti2010comparative,avila2009recent}. The estimation procedure becomes, again, iterative and, therefore, a rapid convergence
to an estimation error that is lower than what PMUs alone can provide, is crucial to render these hybrid systems useful. The goal of our paper is to provide a criterion to ensure the best of both worlds: greater accuracy and faster convergence for the hybrid system.
Before describing our contribution we briefly review the criteria that have been used thus far to select PMUs placements.

\subsubsection{Related Works}
The primary concern of measurement system design for PSSE is to guarantee the {\it observability} of the grid so that the state can be solved without ambiguities, which typically depends on the {\it number of measurements} available. Furthermore, it is also essential that the device {\it locations} are chosen such that they do not result in the formation of {\it critical measurements}, whose existence makes the system susceptible to inobservability due to measurements loss.

Therefore, conventional placement designs typically aim at minimizing the number and/or the cost of the sensors under various observability constraints, see e.g.\cite{gou2008optimal,gou2001improved,madtharad2003measurement,emami2010robust,baldwin1993power,bei2005optimal}. More specifically, \cite{gou2008optimal,gou2001improved} ensure observability by enforcing the algebraic invertibility of the linearized load-flow models or enhancing the numerical condition of the linear model \cite{madtharad2003measurement}. By treating the grid as a graph \cite{clements1990observability}, the schemes in \cite{bei2005optimal,baldwin1993power,emami2010robust} guarantee topological observability, corresponding to the requirement that all the buses have a path connected to at least one device. In general, algebraic observability implies topological observability for linear load-flow models but not vice versa \cite{baldwin1993power}. To suppress or eliminate critical measurements, \cite{magnago2000unified, chakrabarti2008optimal, chakrabarti2009placement, yehia2001pc, baran1995meter} propose placements that guarantee system observability even in case of device/branch outages, or bad data injections. These methods usually take a divide-and-conquer approach and include multiple stages. Specifically, the first stage determines a measurement set with fixed candidates (or size) by cost minimization, and then reduces (or selects) measurements within this set to ensure the topological observability. Numerical techniques such as genetic algorithms \cite{milosevic2003nondominated,aminifar2009optimal}, simulated-annealing \cite{baldwin1993power} and integer linear programming \cite{dua2008optimal} have also been applied in similar placement problems.


In addition to observability, authors have also targeted improvements in the estimation performance. For example, \cite{park1988design} minimizes a linear cost of individual devices subject to a total error constraint, while \cite{zhang2010observability} uses a two-stage approach that first guarantees topological observability and then refines the placement to improve estimation accuracy. In \cite{zhu2009application,asprou2011optimal}, instead, PMUs are placed iteratively on buses with the highest error (individual or sum), until a budget is met. A greedy method was proposed in \cite{li2011phasor}  for PMU placement by minimizing the estimation errors of the augmented PSSE using voltage and linearized power injection measurements. A similar problem is solved in \cite{kekatos2011convex} via convex relaxation and in \cite{li2012information} by maximizing the mutual information between sensor measurements and state vector.



\subsubsection{Motivation and Contributions}
The PMU placements algorithms in the literature  target typically a single specific criterion, observability or accuracy (see the reviews \cite{manousakis2012taxonomy,yuill2011optimal}). It was pointed out recently in \cite{manousakis2012taxonomy} that these objectives should be considered jointly, because designs for pure observability often have multiple solutions (e.g., \cite{chakrabarti2008optimal}) and they are  insufficient to provide accurate estimates. In this paper, we revisit this problem from a unified perspective.
Specifically we jointly consider observability, critical measurements, device outages and failures, estimation performance together with another important criterion that is oftentimes neglected, which is the convergence of the Gauss-Newton (GN) algorithm typically used in state estimation solvers.
Our contribution is: 1) the derivation of the {\it Convergence-Observability-Performance} (COP) metric to evaluate the numerical properties, estimation performance and reliability for a given placement and 2) the solution of the optimum COP metric placement as a semidefinite program (SDP) with integer constraints. We also show that the optimization  can be solved  through a convex transformation, relaxing the integer constraints.  The performance of our design framework is compared successfully with alternatives in simulations.

\subsubsection{Notations}
We used the following notations:
\begin{itemize}
	\item $\mathrm{i}$: imaginary unit and $\mathbb{R}$ and $\mathbb{C}$: real and complex numbers.
	\item $\Re\{\cdot\}$ and $\Im\{\cdot\}$: the real and imaginary part of a number.
	\item $\mathbf{I}_N$: an $N\times N$ identity matrix.
	\item $\|\mathbf{A}\|$ and $\|\mathbf{A}\|_F$ are the $2$-norm\footnote{The $2$-norm of a matrix is the maximum of the absolute value of the eigenvalues and the $2$-norm of a vector $\mathbf{x}\in\mathbb{R}^N$ is $\|\mathbf{x}\|=\sqrt{\sum_{n=1}^N x_n^2}$.} and $F$-norm of a matrix.
	\item $\mathrm{vec}(\mathbf{A})$ is the vectorization of a matrix $\mathbf{A}$.
	\item $\mathbf{A}^T$, $\mathrm{Tr}(\mathbf{A})$, $\lambda_{\min}(\mathbf{A})$ and $\lambda_{\max}(\mathbf{A})$: transpose, trace, minimum and maximum eigenvalues of matrix $\mathbf{A}$.
	\item $\otimes$ is the Kronecker product and $\mathbb{E}[\cdot]$ means expectation.
	\item Given two symmetric matrices $\mathbf{A}$ and $\mathbf{B}$, expressions $\mathbf{A}\succeq \mathbf{B}$ and $\mathbf{A}\succ \mathbf{B}$ represent that the matrix $(\mathbf{A}-\mathbf{B})$ is positive semidefinite and positive definite respectively (i.e., its eigenvalues are all non-negative or positive).
\end{itemize}

\section{Measurement Model and State Estimation}
We consider a power grid with $N$ {\it buses} (i.e., substations), representing interconnections, generators or loads. They are denoted by the set $\mathcal{N}\triangleq\{1,\cdots,N\}$, which form the edge set $\mathcal{E} \triangleq \{\{n,m\}\}$ of cardinality $|\mathcal{E}|=L$, with $\{n,m\}$ denoting the transmission line between $n$ and $m$. Furthermore, we define $\mathcal{N}(n)\triangleq \{m : \{n,m\}\in\mathcal{E}\}$ as the neighbor of bus $n$ and let $L_n = |\mathcal{N}(n)|$. Control centers collect measurements on certain buses and transmission lines to estimate the state of the power system, i.e., the voltage phasor $V_n\in\mathbb{C}$ at each bus $n\in\mathcal{N}$. In this paper, we consider the Cartesian coordinate representation using the real and imaginary components of the complex voltage phasors $\mathbf{v}=[\Re\{V_1\},\cdots,\Re\{V_N\},\Im\{V_1\},\cdots,\Im\{V_N\}]^T$. This representation facilitates our derivations because it expresses PMU measurements as a linear mapping and SCADA measurements as quadratic forms of the state $\mathbf{v}$ (see \cite{lavaei2010zero}).


\vspace{-0.4cm}
\subsection{Hybrid State Estimation}

The measurement set used in PSSE contains SCADA measurements and PMU measurements from the WAMS.
Since there are 2 complex nodal variables at each bus (i.e., power injection and voltage), and 4 complex line measurements (i.e., power flow and current), the total number of variables is $2M$, considering real and imaginary parts, where $M=2N+4L$ is the total number of either the PMU (i.e., voltage and current) or SCADA (i.e., power injection and flow) variables in the ensemble. Thus, the  ensemble of variables can be partitioned into four vectors $\mathbf{z}=[\mathbf{z}_{\mathcal{V}}^T,\mathbf{z}_{\mathcal{C}}^T,\mathbf{z}_{\mathcal{I}}^T,\mathbf{z}_{\mathcal{F}}^T]^T$, containing the $2N$ voltage phasor $\mathbf{z}_{\mathcal{V}}$ and power injection vector $\mathbf{z}_{\mathcal{I}}$ at bus $n\in\mathcal{N}$, the $4L$ current phasor $\mathbf{z}_{\mathcal{C}}$ and power flow vector $\mathbf{z}_{\mathcal{F}}$ on line $\{n,m\}\in\mathcal{E}$ at bus $n$. Note that the subscripts $\mathcal{V}$, $\mathcal{C}$, $\mathcal{I}$ and $\mathcal{F}$ are chosen to indicate ``voltage'', ``current'', ``injection'' and ``flow'' respectively.
The power flow equations $\mathbf{f}_{\mathcal{V}}(\mathbf{v})$,  $\mathbf{f}_{\mathcal{C}}(\mathbf{v})$,  $\mathbf{f}_{\mathcal{I}}(\mathbf{v})$,  $\mathbf{f}_{\mathcal{F}}(\mathbf{v})$ are specified in Appendix \ref{power_flow_equations_appendix} for different types of measurements. Letting $\mathbf{v}_{\textrm{\tiny true}}$ be the true system state,
we have
\begin{align}\label{meas-model_all}
	\mathbf{z} = \mathbf{f}(\mathbf{v}_{\textrm{\tiny true}}) + \mathbf{r},
\end{align}
where $\mathbf{r}=[\mathbf{r}_{\mathcal{V}}^T,\mathbf{r}_{\mathcal{C}}^T,\mathbf{r}_{\mathcal{I}}^T,\mathbf{r}_{\mathcal{F}}^T]^T$ is the aggregate measurement noise vector, with $\mathbb{E}\{\mathbf{r}\}=\mathbf{0}$ and a covariance matrix $\mathbf{R}\triangleq\mathbb{E}\{\mathbf{r}\mathbf{r}^T\}$, and $\mathbf{f}(\mathbf{v})=[\mathbf{f}_{\mathcal{V}}^T(\mathbf{v}),\mathbf{f}_{\mathcal{C}}^T(\mathbf{v}),\mathbf{f}_{\mathcal{I}}^T(\mathbf{v}),\mathbf{f}_{\mathcal{F}}^T(\mathbf{v})]^T$ refer to the aggregate power flow equations.

The actual measurements set used in PSSE is a subset of $\mathbf{z}$ in \eqref{meas-model_all}, depending on the SCADA and WAMS sensors deployment. Specifically, we introduce a $2M\times 2M$ mask
\begin{align}\label{mask}
	\mathbf{J}_{\mathcal{A}}  \triangleq \mathrm{diag}[\mathbf{J}_{\mathcal{V}},\mathbf{J}_{\mathcal{C}},\mathbf{J}_{\mathcal{I}},\mathbf{J}_{\mathcal{F}}],
\end{align}
where $\mathbf{J}_{\mathcal{V}}$, $\mathbf{J}_{\mathcal{C}}$, $\mathbf{J}_{\mathcal{I}}$ and $\mathbf{J}_{\mathcal{F}}$ are the diagonal masks for each measurement type, having $1$ on its diagonal if that measurement is chosen. Applying this mask on the ensemble $\mathbf{z}$ gives
\begin{align}\label{meas-model}
	\mathbf{J}_{\mathcal{A}}\mathbf{z} = \mathbf{J}_{\mathcal{A}}\mathbf{f}(\mathbf{v}_{\textrm{\tiny true}}) + \mathbf{J}_{\mathcal{A}}\mathbf{r}.
\end{align}
The vector $\mathbf{J}_{\mathcal{A}}\mathbf{z} $ are the measurements used in estimation, having non-zero entries selected by $\mathbf{J}_{\mathcal{A}}$ and zero otherwise.

Assuming the noise is uncorrelated with constant variances for each type,  $\mathbf{R}=\mathrm{diag}\left[\sigma^2_\mathcal{V}\mathbf{I},\sigma_\mathcal{C}^2\mathbf{I},\sigma_\mathcal{I}^2\mathbf{I},\sigma_\mathcal{F}^2\mathbf{I}\right]$. The state is: 
\begin{align}\label{centralized_SE}
	 \mathbf{v}_{\textrm{\tiny est}} = \underset{\mathbf{v}\in\mathbb{V}}{\argmin}&
	 ~~ \|\mathbf{z}_{\mathcal{A}}- \mathbf{f}_{\mathcal{A}}(\mathbf{v})\|^2,
\end{align}
where $\mathbf{z}_{\mathcal{A}}=\mathbf{R}^{-\frac{1}{2}}\mathbf{J}_{\mathcal{A}}\mathbf{z}$ and $\mathbf{f}_{\mathcal{A}}(\mathbf{v}) = \mathbf{R}^{-\frac{1}{2}}\mathbf{J}_{\mathcal{A}}\mathbf{f}(\mathbf{v})$ are the re-weighted versions of $\mathbf{z}$ and $\mathbf{f}(\mathbf{v})$ by the covariance $\mathbf{R}$, and $\mathbb{V}\triangleq(0,V_{\max}]^{2N}$ is the state space. Without loss of generality, the GN algorithm is usually used to solve \eqref{centralized_SE} for the state.

Although there are variants of the GN algorithm, we study the most basic form of GN updates
\begin{align}\label{central_update}
	\mathbf{v}^{k+1} &= \mathbf{v}^k + \mathbf{d}^k,\quad k = 1,2,\cdots
\end{align}
with a chosen {\it initializer} $\mathbf{v}^0$ and the iterative {\it descent}
\begin{align}\label{central_descent}	
	 \mathbf{d}^k &=\left[\mathbf{F}_{\mathcal{A}}^T(\mathbf{v}^k)\mathbf{F}_{\mathcal{A}}(\mathbf{v}^k)\right]^{-1}\mathbf{F}_{\mathcal{A}}^T(\mathbf{v}^k)\left[\mathbf{z}_{\mathcal{A}}-\mathbf{f}_{\mathcal{A}}(\mathbf{v}^k)\right],
\end{align}
where $\mathbf{F}_{\mathcal{A}}^T(\mathbf{v})\mathbf{F}_{\mathcal{A}}(\mathbf{v})$ is called the {\it gain matrix} and
$\mathbf{F}_{\mathcal{A}}(\mathbf{v}) 
=\mathbf{R}^{-\frac{1}{2}}\mathbf{J}_{\mathcal{A}}{\mathrm{d}\mathbf{f}(\mathbf{v})}/{\mathrm{d}\mathbf{v}^T}$ is the Jacobian corresponding to the selected measurements.
The full Jacobian $\mathbf{F}(\mathbf{v})\triangleq {\mathrm{d}\mathbf{f}(\mathbf{v})}/{\mathrm{d}\mathbf{v}^T}$ is computed in Appendix \ref{power_flow_equations_appendix}.

\vspace{-0.4cm}
\subsection{Gain Matrix and the PMU Placement}
The design of $\mathbf{J}_{\mathcal{A}}$ is crucial for the success of PSSE because $\mathbf{J}_{\mathcal{A}}$ affects the condition number of the gain matrix in \eqref{central_descent}, which determines the observability of the grid, the stability of the update of state estimates and the ultimate accuracy of the estimates (see the corresponding connections between the gain matrix and these issues in Section \ref{observability_issue}, \ref{convergence_issue} and \ref{performance_issue}). The goal of this subsection is to express explicitly the dependency of the gain matrix on the PMU placement. Since SCADA systems have been deployed for decades, we assume that SCADA measurements are given so that $\mathbf{J}_{\mathcal{I}},\mathbf{J}_{\mathcal{F}}$ are fixed, and focus on designing the PMU placement $\mathbf{J}_{\mathcal{V}},\mathbf{J}_{\mathcal{C}}$. We consider the case where each installed PMU captures the voltage and all incident current measurements on that bus as in \cite{kekatos2011convex,baldwin1993power}, so that the current selections $\mathbf{J}_{\mathcal{C}}$ depend entirely on $\mathbf{J}_{\mathcal{V}}$. Therefore, we define the PMU placement vector as
\begin{align}
	\bdsb{\mathcal{V}}\triangleq[\mathcal{V}_1,\cdots,\mathcal{V}_N]^T,\quad \mathcal{V}_n\in\{0,1\},
\end{align}
indicating if the $n$-th bus has a PMU and $\mathbf{J}_{\mathcal{V}}=\mathbf{I}_2\otimes\mathrm{diag}(\bdsb{\mathcal{V}})$, while the the power injection and power flow measurement placements are given by $\bdsb{\mathcal{I}}\triangleq[\mathcal{I}_n]_{N\times 1}$ and $\bdsb{\mathcal{F}}\triangleq[\mathcal{F}_{nm}]_{N\times N}$ with $\mathcal{I}_n$ and $\mathcal{F}_{nm}\in\{0,1\}$
to indicate whether the injection at bus $n$ and power flow on line $\{n,m\}$ measured at bus $n$ are present in the PSSE. Similarly we have $\mathbf{J}_{\mathcal{I}}=\mathbf{I}_2\otimes\mathrm{diag}(\bdsb{\mathcal{I}})$ and $\mathbf{J}_{\mathcal{F}}=\mathbf{I}_2\otimes\mathrm{diag}[\mathrm{vec}(\bdsb{\mathcal{F}})]$.
Finally, given an arbitrary state $\mathbf{v}$, the gain matrix in \eqref{central_descent} can be decomposed into two components
\begin{align}\label{GN_Hessian}
	\mathbf{F}_{\mathcal{A}}^T(\mathbf{v})\mathbf{F}_{\mathcal{A}}(\mathbf{v})
	&=
	 \mathcal{P}(\bdsb{\mathcal{V}})+\mathcal{S}(\mathbf{v}\mathbf{v}^T,\bdsb{\mathcal{I}},\bdsb{\mathcal{F}})
\end{align}
using matrices $\mathbf{H}_{I,n}$, $\mathbf{H}_{J,n}$, $\mathbf{N}_{P,n}$, $\mathbf{N}_{Q,n}$, $\mathbf{E}_{P,nm}$ and $\mathbf{E}_{Q,nm}$ given explicitly by \eqref{Jacobian} in Appendix \ref{power_flow_equations_appendix}. The exact expression for each component can be analytically written as: \\
{\bf $\mathbf{(1)}$ PMU data} $\mathcal{P}(\bdsb{\mathcal{V}}):\mathbb{R}^{2N} \rightarrow \mathbb{R}^{2N\times 2N}$
\begin{align*}
	\mathcal{P}(\bdsb{\mathcal{V}})
	&= \sum_{n=1}^N  \mathcal{V}_n \left(\frac{\mathbf{I}_2\otimes\mathbf{e}_n\mathbf{e}_n^T}{\sigma_{\mathcal{V}}^2}
	+  \frac{\mathbf{H}_{I,n}^T\mathbf{H}_{I,n} + \mathbf{H}_{J,n}^T\mathbf{H}_{J,n}}{\sigma_{\mathcal{C}}^2}
	\right)
\end{align*}
{\bf $\mathbf{(2)}$ SCADA data} $\mathcal{S}(\mathbf{V},\bdsb{\mathcal{I}},\bdsb{\mathcal{F}}):\mathbb{R}^{2N\times 2N} \rightarrow \mathbb{R}^{2N\times 2N}$
\begin{align*}
	\mathcal{S}(\mathbf{V}&,\bdsb{\mathcal{I}},\bdsb{\mathcal{F}})
	= \sum_{n=1}^N \frac{\mathcal{I}_n}{\sigma_{\mathcal{I}}^2}\left(\mathbf{N}_{P,n}+\mathbf{N}_{P,n}^T\right)^T\mathbf{V}\left(\mathbf{N}_{P,n}+\mathbf{N}_{P,n}^T\right)\\
	&+ \sum_{n=1}^N \frac{\mathcal{I}_n}{\sigma_{\mathcal{I}}^2}\left(\mathbf{N}_{Q,n}+\mathbf{N}_{Q,n}^T\right)^T\mathbf{V}\left(\mathbf{N}_{Q,n}+\mathbf{N}_{Q,n}^T\right)\nonumber\\
	&+ \sum_{n,l} \frac{\mathcal{F}_{nm}}{\sigma_{\mathcal{F}}^2}\left(\mathbf{E}_{P,nm}+\mathbf{E}_{P,nm}^T\right)^T\mathbf{V}\left(\mathbf{E}_{P,nm}+\mathbf{E}_{P,nm}^T\right)\nonumber\\
 	&+\sum_{n,l} \frac{\mathcal{F}_{nm}}{\sigma_{\mathcal{F}}^2}\left(\mathbf{E}_{Q,nm}+\mathbf{E}_{Q,nm}^T\right)^T\mathbf{V}\left(\mathbf{E}_{Q,nm}+\mathbf{E}_{Q,nm}^T\right),\nonumber
\end{align*}	
where $\mathbf{V}=\mathbf{v}\mathbf{v}^T$. The derivations are tedious but straightforward from \eqref{GN_Hessian} and \eqref{Jacobian} and thus omitted due to limited space.

Note that although the PMU placement design is the focus of this paper, we also consider its complementary benefits on the overall reliability of the PSSE mostly based on SCADA data, by showing how PMUs can eliminate critical measurements issues, as explained in Section \ref{optimal_PMU_placement}.



\section{Measurement Placement Design}
In this section, we address three important aspects of the placement design as a prequel to the comprehensive metric for PMU placement proposed in Section \ref{combined_metric}, including observability, convergence, and accuracy, which are all derived with respect to the task of performing state estimation. We call this comprehensive metric the {\it COP metric}, which is an abbreviation for {\it Convergence}, {\it Observability} and {\it Performance}. In Section \ref{combined_metric}, we further derive how the PMU placement $\bdsb{\mathcal{V}}$ affects this metric analytically. Later in Section \ref{optimal_PMU_placement}, we optimize the placement using this metric under observability constraints in case of measurement loss or device malfunction.

\subsection{Observability}\label{observability_issue}
As mentioned previously, observability analysis is the foundation for all PSSE because it guarantees that the selected measurements are sufficient to solve for the state without ambiguity. There are two concepts associated with this issue, which are the {\it topological observability} and the {\it numerical (algebraic) observability}. Topological observability, in essence, studies the measurement system as a graph and determines whether the set of nodes corresponding to the measurement set in PSSE constitute a dominating set of the grid (i.e., each node is a direct neighbor of the nodes that provide the measurement set). Numerical observability, instead, is typically based on the linearized decoupled load flow model \cite{stott1974fast}, and recently the PMU model \cite{kekatos2011convex,bei2005optimal,nuqui2005phasor,dua2008optimal,emami2010robust,gou2008optimal,abbasy2009unified,chakrabarti2008optimal,chakrabarti2009placement,baldwin1993power,milosevic2003nondominated,aminifar2009optimal}, focused on the algebraic invertibility of the PSSE problem. Although the topological observability bears different mathematical interpretations than numerical observability, oftentimes they are both valid measures if the admittance matrix does not suffer from singularity \cite{baldwin1993power,clements1990observability}.


\begin{rmk}\label{observability}
{\bf (observability)} Using the gain matrix expression in \eqref{GN_Hessian}, the observability can be guaranteed by having
\begin{align}\label{beta}
	\beta(\bdsb{\mathcal{V}}) 
	&=\underset{\mathbf{v}\in\mathbb{V}}{\inf}~\lambda_{\min}
	\left[\mathcal{P}(\bdsb{\mathcal{V}})+\mathcal{S}(\mathbf{v}\mathbf{v}^T,\bdsb{\mathcal{I}},\bdsb{\mathcal{F}})\right]>0.
\end{align}
\end{rmk}
Given a fixed SCADA placement $\bdsb{\mathcal{I}}$ and $\bdsb{\mathcal{F}}$, the value of $\beta(\bdsb{\mathcal{V}})$ depends on the PMU placement $\bdsb{\mathcal{V}}$ which should, therefore, be designed so that $\beta(\bdsb{\mathcal{V}})>0$. Although observability guarantees the existence and uniqueness of the PSSE solution, it does not imply that the state estimate obtained from the GN algorithm \eqref{central_update} is the correct state estimate, since the solution could be a local minimum. This is especially the case when the initializer $\mathbf{v}^0$ is not chosen properly. Thus, observability is a meaningful criterion only if one assumes successful convergence, as discussed next.

\subsection{Convergence}\label{convergence_issue}
The convergence of state estimation to the correct estimate $\mathbf{v}_{\textrm{\tiny est}}$  using the AC power flow models in \eqref{centralized_SE} has been a critical issue in PSSE. With SCADA measurements, state estimation based on the AC power flow model in \eqref{centralized_SE} is in general non-convex and there might be multiple {\it fixed points} $\mathbf{v}^\star$ of the update in \eqref{central_update} that stop the iterate $\mathbf{v}^k$ from progressing towards the correct estimate $\mathbf{v}_{\textrm{\tiny est}}$. Let the set of fixed points $\mathbf{v}^\star$ be
\begin{align}\label{fixed_point}
     \mathbb{V}^\star = \left\{\mathbf{v}\in\mathbb{R}^{2N}:\mathbf{F}_{\mathcal{A}}^T(\mathbf{v})\big[\mathbf{z}_{\mathcal{A}}-\mathbf{f}_{\mathcal{A}}(\mathbf{v})\big] = \mathbf{0}\right\}.
\end{align}
Clearly, the correct estimate $\mathbf{v}_{\textrm{\tiny est}}$ of \eqref{centralized_SE} is in this set $\mathbf{v}_{\textrm{\tiny est}}\in\mathbb{V}^\star$. As a result, there are two convergence issues to address, including a proper {\it initialization} $\mathbf{v}^0$ and the stabilization of the error $\left\|\mathbf{v}^k-\mathbf{v}_{\textrm{\tiny est}}\right\|$ made relative the global estimate $\mathbf{v}_{\textrm{\tiny est}}\in\mathbb{V}^\star$ instead of other fixed points $\mathbf{v}^\star\in\mathbb{V}^\star$. Because an accurate measurement of the state can be directly obtained by the PMU device, it is natural to exploit such measurements as a good initializer to start the GN algorithm. In the following, we first explain the PMU-assisted initialization scheme, and then present the error dynamics analysis.

We propose to choose the initializer $\mathbf{v}^0 $ to match PMU measurements on PMU-instrumented buses, with the rest provided by an arbitrary initializer $\mathbf{v}_{\textrm{\tiny prior}}$. The initializer is expressed as 
\begin{align}\label{x0}
	\mathbf{v}^0
	= \mathbf{J}_{\mathcal{V}}\mathbf{z}_{\mathcal{V}} + (\mathbf{I}_{2N}-\mathbf{J}_{\mathcal{V}})\mathbf{v}_{\textrm{\tiny prior}},	
\end{align}
where $\mathbf{v}_{\textrm{\tiny prior}}$ is a stale estimate or nominal profile, and $\mathbf{J}_{\mathcal{V}}=\mathbf{I}_2\otimes\mathrm{diag}(\bdsb{\mathcal{V}})$. Given a placement $\bdsb{\mathcal{V}}$, we analyze the error dynamics of the update in \eqref{central_update}, which examines the iterative error progression over iterations as a result of the placement.
\begin{lem}\label{lem_error_recursion}
Defining the iterative error at the $k$-th update as $\rho_k=\left\|\mathbf{v}^k-\mathbf{v}_{\textrm{\tiny est}}\right\|$, we have the following error dynamics
\begin{align}\label{error_recursion}
	\rho_{k+1}
	&\leq  \frac{1}{2}\sqrt{\frac{\phi_k}{\beta(\bdsb{\mathcal{V}})}} \rho_k^2 + \frac{\epsilon\sqrt{2\phi_k}}{\beta(\bdsb{\mathcal{V}})} \rho_k.
\end{align}
$\epsilon= \|\mathbf{z}_{\mathcal{A}}-\mathbf{f}_{\mathcal{A}}(\mathbf{v}_{\textrm{\tiny est}})\|$ is the optimal reconstruction error and
\begin{align}\label{phik}
	\phi_k=\frac{\left(\mathbf{v}^k-\mathbf{v}_{\textrm{\tiny est}}\right)^T\mathbf{M}\left(\mathbf{v}^k-\mathbf{v}_{\textrm{\tiny est}}\right)}{\left(\mathbf{v}^k-\mathbf{v}_{\textrm{\tiny est}}\right)^T\left(\mathbf{v}^k-\mathbf{v}_{\textrm{\tiny est}}\right)}
\end{align}
is a Rayleigh quotient of the matrix $\mathbf{M}\triangleq  \mathcal{S}(\mathbf{I}_{2N},\bdsb{\mathcal{I}},\bdsb{\mathcal{F}})$, equal to  $\mathcal{S}(\mathbf{V},\bdsb{\mathcal{I}},\bdsb{\mathcal{F}})$ in  \eqref{GN_Hessian} with $\mathbf{V}=\mathbf{I}_{2N}$.
\end{lem}
\begin{proof}
	See Appendix \ref{proof_lem_error_recursion}.
\end{proof}
Lemma \ref{lem_error_recursion} describes the coupled dynamics of the error $\rho_k$ and the quantity $\phi_k$. However, we are only interested in the dynamics of $\rho_k$, which govern how fast the state estimate reaches the ultimate accuracy. Let us denote an upper bound\footnote{Note that the worst case of this upper bound is clearly $\phi(\bdsb{\mathcal{V}})\leq \lambda_{\max}(\mathbf{M})$.} $\phi(\bdsb{\mathcal{V}})$ for all $\phi_k$ that depends on $\bdsb{\mathcal{V}}$. From Lemma \ref{lem_error_recursion}, it follows:

\begin{thm}\label{cor_convergence}
\cite[Theorem 1]{li2012convergence} Given an upper bound $ \phi(\bdsb{\mathcal{V}})\geq \phi_k$ for all $k$ and suppose $\epsilon \sqrt{2\phi(\bdsb{\mathcal{V}})} < \beta(\bdsb{\mathcal{V}})$, then the algorithm converges $\underset{k\rightarrow\infty}{\lim}~\rho_k = 0$ if the initialization satisfies
\begin{align}
	\rho_0 \leq  2\sqrt{\frac{\beta(\bdsb{\mathcal{V}})}{\phi(\bdsb{\mathcal{V}})}} - \frac{2\sqrt{2}\epsilon}{\sqrt{\beta(\bdsb{\mathcal{V}})}}.
\end{align}
\end{thm}
%
\begin{rmk}\label{convergence}
{\bf (convergence)} With a low optimal reconstruction error $\epsilon\approx 0$, the implications of Lemma \ref{lem_error_recursion} and Theorem \ref{cor_convergence} are:
\begin{itemize}
	\item the sensitivity to initialization is determined by the radius
	\begin{align}
		\rho_0 \leq 2\sqrt{\frac{\beta(\bdsb{\mathcal{V}})}{\phi(\bdsb{\mathcal{V}})}}
	\end{align}	
	\item the error converges quadratically at an asymptotic rate
	\begin{align}
		\underset{k\rightarrow \infty}{\lim} \frac{\rho_{k+1}}{\rho_k^2}
		&\leq  \frac{1}{2}\sqrt{\frac{\phi(\bdsb{\mathcal{V}})}{\beta(\bdsb{\mathcal{V}})}}.
	\end{align}
\end{itemize}
\end{rmk}
In other words, the larger is the ratio $\beta(\bdsb{\mathcal{V}})/\phi(\bdsb{\mathcal{V}})$, the larger is the radius of convergence and the faster the algorithm converges. Similar to the observability metric in Remark \ref{observability}, the convergence is  determined by the PMU placement $\bdsb{\mathcal{V}}$. This is confirmed by simulations in Section \ref{simulations}, when $\mathbf{v}^0$ is mildly perturbed. The state estimate diverges drastically to a wrong point if the PMU placement is not chosen carefully and furthermore, in cases where the algorithm converges, the PMU placement significantly affects the rate of convergence.

What remains to be determined is the bound $\phi(\bdsb{\mathcal{V}})$.  One simple option is to bound the Rayleigh quotient $\phi_k$ for each iteration $k$ with the largest eigenvalue $\lambda_{\max}(\mathbf{M})$. However, this is a pessimistic bound that ignores the dependency of $\phi_k$  on $\mathbf{v}^k$, due to the initialization $\mathbf{v}^0$ in \eqref{x0}.
In the proposition below, we motivate the following choice of the upper bound.
\begin{proposition}\label{prop_phi}
The bound $\phi(\bdsb{\mathcal{V}})$ can be approximated by
\begin{align}\label{omega_OPT}
    \phi(\bdsb{\mathcal{V}}) ~\approx~ \lambda_{\max}\left[\left(\mathbf{I}_{2N}-\mathbf{J}_{\mathcal{V}}\right)^T\mathbf{M}\left(\mathbf{I}_{2N}-\mathbf{J}_{\mathcal{V}}\right)\right].
\end{align}
\end{proposition}
\begin{proof}
	See Appendix \ref{proof_prop_phi}.
\end{proof}
%

%


\vspace{-0.3cm}
\subsection{Performance (Accuracy)}\label{performance_issue}
\vspace{-0.1cm}
Given Remark \ref{observability} and \ref{convergence} for observability and convergence, we proceed to discuss the accuracy of the state estimator. This is evaluated by the error between the iterate $\mathbf{v}^k$ and the true state $\mathbf{v}_{\textrm{\tiny true}}$, which can be bounded by the triangular inequality
\begin{align}
	\left\|\mathbf{v}^k-\mathbf{v}_{\textrm{\tiny true}}\right\| \leq \left\|\mathbf{v}^k-\mathbf{v}_{\textrm{\tiny est}}\right\|+ \left\|\mathbf{v}_{\textrm{\tiny est}} - \mathbf{v}_{\textrm{\tiny true}}\right\|.
\end{align}
If the iterate $\mathbf{v}^k$ converges stably to the correct estimate ${\lim}_{k\rightarrow\infty}~\mathbf{v}^k = \mathbf{v}_{\textrm{\tiny est}}$, the error can be bounded accordingly by
\begin{align}
	{\lim}_{k\rightarrow\infty}~\left\|\mathbf{v}^{k+1} - \mathbf{v}_{\textrm{\tiny true}}\right\| \leq \left\|\mathbf{v}_{\textrm{\tiny est}} - \mathbf{v}_{\textrm{\tiny true}}\right\|.
\end{align}	
If the noise $\mathbf{r}$ in \eqref{meas-model_all} is Gaussian, the estimate $\mathbf{v}_{\textrm{\tiny est}}$ given by \eqref{centralized_SE} is the Maximum Likelihood (ML) estimate. According to classic estimation theory \cite{kayfundamentals}, the mean square error (MSE) of the ML estimates reaches the Cram\'{e}r-Rao~Bound (CRB) asymptotically given sufficient measurements
\begin{align}\label{CRB}
	\mathbb{E}\left[\left\|\mathbf{v}_{\textrm{\tiny est}}-\mathbf{v}_{\textrm{\tiny true}}\right\|^2\right]
	= \mathrm{Tr}\left[\left(\mathbf{F}_{\mathcal{A}}^T(\mathbf{v}_{\textrm{\tiny true}})\mathbf{F}_{\mathcal{A}}(\mathbf{v}_{\textrm{\tiny true}})\right)^{-1}\right]
\end{align}
where the expectation is with respect to the noise distribution $\mathbf{r}$, and the gain matrix $\mathbf{F}_{\mathcal{A}}^T(\mathbf{v}_{\textrm{\tiny true}})\mathbf{F}_{\mathcal{A}}(\mathbf{v}_{\textrm{\tiny true}})$ evaluated at the true state $\mathbf{v}_{\textrm{\tiny true}}$ is the Fisher Information Matrix (FIM).

Many placement designs focus on lowering the CRB in different ways. Specifically, the $A$-, $M$- and {\it accuracy} designs\footnote{There is also a $D$-{\it optimal} in \cite{li2011phasor,kekatos2011convex}, which minimizes the logarithm of the determinant of the FIM, we omit it because it shares less in common with other related works. In simulations, we compare our design only with the {\it accuracy} design because of the common objective in maximizing $\beta(\bdsb{\mathcal{V}})$. Other $A$-, $D$- and $M$-{\it optimal} designs provide similar performances and hence are not repeated in simulations.} in \cite{li2011phasor,kekatos2011convex} focus on maximizing the trace, the minimum diagonal element, and the minimum eigenvalue of the FIM in \eqref{CRB} respectively. Other existing works considering estimation accuracy optimize their designs with respect to the FIM in an ad-hoc manner. For example, \cite{park1988design} minimizes the cost of PMU deployment under a total error constraint on the trace of the FIM, while \cite{zhang2010observability,zhu2009application} are similar to the $M$-{\it optimal} design in picking heuristically the locations by pinpointing the maximum entry in the FIM.
\begin{rmk}\label{performance}
{\bf (performance)} Given a specific PMU placement $\bdsb{\mathcal{V}}$, the MSE of the state estimation is upper bounded as
\begin{align*}
    \mathbb{E}\left[\left\|\mathbf{v}_{\textrm{\tiny est}}-\mathbf{v}_{\textrm{\tiny true}}\right\|^2\right]
	= \mathrm{Tr}\left[\left(\mathbf{F}_{\mathcal{A}}^T(\mathbf{v})\mathbf{F}_{\mathcal{A}}(\mathbf{v})\right)^{-1}\right]
	\leq \frac{2N}{\beta(\bdsb{\mathcal{V}})},
\end{align*}
therefore $\beta(\bdsb{\mathcal{V}})$ is an important metric for PMU placements from the observability and performance perspective.
\end{rmk}
\begin{proof}
	See Appendix \ref{proof_cond_perf}.
\end{proof}
%





\vspace{-0.2cm}

\section{Optimal PMU placement via the COP metric}\label{optimal_PMU_placement}\label{combined_metric}
\vspace{-0.1cm}


Based on Remark \ref{observability}, \ref{convergence} and \ref{performance}, we are ready to introduce our {\it Convergence-Observability-Performance (COP) metric}
\begin{equation}\label{JAC}
	\rho(\bdsb{\mathcal{V}}) = \frac{\beta(\bdsb{\mathcal{V}})}{\phi(\bdsb{\mathcal{V}})},\quad (\mathrm{COP~metric})
\end{equation}
where $\beta(\bdsb{\mathcal{V}})$ is defined in \eqref{beta} and $\phi(\bdsb{\mathcal{V}})$ is the upper bound (used in Theorem \ref{cor_convergence}) of the Rayleigh quotient $\phi_k$ in Lemma \ref{lem_error_recursion}. In fact, it is seen from Remark \ref{observability}, \ref{convergence} and \ref{performance} that the greater the value of $\rho(\bdsb{\mathcal{V}})$: 1) the less sensitive PSSE is to initialization; 2) the faster the algorithm converges asymptotically; 3) the observability and performance metric $\beta(\bdsb{\mathcal{V}})$ scales linearly given $\phi(\bdsb{\mathcal{V}})$.
Therefore, we propose to have the PMUs stabilize the algorithm by giving a good initialization and potentially lowering the estimation error. Next, we exploit the dependency of $\beta(\bdsb{\mathcal{V}})$ and $\phi(\bdsb{\mathcal{V}})$ on  $\bdsb{\mathcal{V}}$ to formulate the placement problem.

We have established the expression of $\beta(\bdsb{\mathcal{V}})$ in \eqref{GN_Hessian}, which however requires an exhaustive search $\mathbf{v}\in\mathbb{V}$. For simplicity, the common practice is to replace the search by substituting the nominal initializer $\mathbf{v}_{\textrm{\tiny prior}}$ in \eqref{x0}, where the flat profile is often chosen $\mathbf{v}_{\textrm{\tiny prior}}=[\mathbf{1}_N^T,\mathbf{0}_N^T]^T$ as in \cite{kekatos2011convex}. This leads to
\begin{align}\label{beta_OPT}
    \beta(\bdsb{\mathcal{V}}) \approx \lambda_{\min}\left[\mathcal{P}(\bdsb{\mathcal{V}})+\mathcal{S}(\mathbf{v}_{\textrm{\tiny prior}}\mathbf{v}_{\textrm{\tiny prior}}^T,\bdsb{\mathcal{I}},\bdsb{\mathcal{F}})\right].
\end{align}


Thus, given a budget on the number of PMUs $N_{\textrm{\tiny PMU}}$ and a total cost constraint $C_{\textrm{\tiny PMU}}$, the {\it optimal} design aims at maximizing the COP metric using the expressions in \eqref{beta_OPT} and \eqref{omega_OPT}
\begin{align}\label{PMU_opt}
	\underset{\bdsb{\mathcal{V}}}{\max}~ &~\frac{\lambda_{\min}\left[\mathcal{P}(\bdsb{\mathcal{V}})+\mathcal{S}(\mathbf{v}_{\textrm{\tiny prior}}\mathbf{v}_{\textrm{\tiny prior}}^T,\bdsb{\mathcal{I}},\bdsb{\mathcal{F}})\right]}{\lambda_{\max}\left[\left(\mathbf{I}_{2N}-\mathbf{J}_{\mathcal{V}}\right)^T\mathbf{M}\left(\mathbf{I}_{2N}-\mathbf{J}_{\mathcal{V}}\right)\right]}\\
                         \mathrm{s.t.}  ~ &~\mathbf{J}_{\mathcal{V}} = \mathbf{I}_2\otimes \mathrm{diag}(\bdsb{\mathcal{V}}),~\mathcal{V}_n\in\{0,1\}\nonumber\\
                         & ~ \mathbf{1}_N^T\bdsb{\mathcal{V}} \leq N_{\textrm{\tiny PMU}},~\bdsb{{{c}}}^T\bdsb{\mathcal{V}} \leq C_{\textrm{\tiny PMU}}, \nonumber
\end{align}
where $\bdsb{{{c}}}=[{{c}}_1,\cdots,{{c}}_N]^T$ contains the cost of each PMU.

Note that maximizing the COP metric alone does not necessarily maximize the observability and performance metric $\beta(\bdsb{\mathcal{V}})$, but instead it is providing a sweet spot between having a good initialization and lowering the estimation error. To ensure that the value of $\beta(\bdsb{\mathcal{V}})$ is sufficiently large, we further consider eliminating critical measurements with a tolerance parameter set by the designer $\beta_{\min}>0$ such that $\beta(\bdsb{\mathcal{V}})$ is guaranteed to surpass an acceptable threshold. Another benefit of eliminating critical measurements is to improve bad data detection capability. Therefore in the next subsection, we formulate the PMU placement problem by considering reliability constraints on data redundancy and {\it critical measurements}.

\vspace{-0.3cm}

\subsection{Elimination of Critical Measurements}

Let us denote by $\bdsb{\mathcal{J}}_n$ and $\bdsb{\mathcal{G}}_{nm}$ the failure patterns for power injection and flow measurements, where the $n$-th bus injection or the line flow on $\{n,m\}$ measured at bus $n$ is removed from the existing SCADA measurements $\bdsb{\mathcal{I}}$ and $\bdsb{\mathcal{F}}$. Then given a tolerance parameter $\beta_{\min}>0$ to ensure the numerical observability\footnote{The value of $\beta_{\min}$ is set to be $0.01$ in simulations for all cases.}, the PMU placement optimization is
\begin{align}\label{PMU_opt_relaxed_critical_meas}
	\underset{\bdsb{\mathcal{V}}}{\max} &~\frac{\lambda_{\min}\left[\mathcal{P}(\bdsb{\mathcal{V}})+\mathcal{S}(\mathbf{v}_{\textrm{\tiny prior}}\mathbf{v}_{\textrm{\tiny prior}}^T,\bdsb{\mathcal{I}},\bdsb{\mathcal{F}})\right]}{\lambda_{\max}\left[\left(\mathbf{I}_{2N}-\mathbf{J}_{\mathcal{V}}\right)^T\mathbf{M}\left(\mathbf{I}_{2N}-\mathbf{J}_{\mathcal{V}}\right)\right]},\\
    \mathrm{s.t.} &~\lambda_{\min}\left[\mathcal{P}(\bdsb{\mathcal{V}})+\mathcal{S}(\mathbf{v}_{\textrm{\tiny prior}}\mathbf{v}_{\textrm{\tiny prior}}^T,\bdsb{\mathcal{J}}_n,\bdsb{\mathcal{F}})\right] \geq \beta_{\min},~\forall n\in\mathcal{N}\nonumber\\
			    &~\lambda_{\min}\left[\mathcal{P}(\bdsb{\mathcal{V}})+\mathcal{S}(\mathbf{v}_{\textrm{\tiny prior}}\mathbf{v}_{\textrm{\tiny prior}}^T,\bdsb{\mathcal{I}},\bdsb{\mathcal{G}}_{nm})\right] \geq \beta_{\min},~\forall \{n,m\}\in\mathcal{E}\nonumber\\
			   &~ \mathbf{J}_{\mathcal{V}} = \mathbf{I}_2\otimes \mathrm{diag}(\bdsb{\mathcal{V}}),~\mathcal{V}_n\in\{0,1\}\nonumber\\
			    &~ \mathbf{1}_N^T\bdsb{\mathcal{V}} \leq N_{\textrm{\tiny PMU}},~\bdsb{{{c}}}^T\bdsb{\mathcal{V}} \leq C_{\textrm{\tiny PMU}}.\nonumber
\end{align}

\begin{rmk}
The constraints above can be easily extended to cover multiple failures by incorporating corresponding outage scenarios $\bdsb{\mathcal{J}}_n$ and $\bdsb{\mathcal{G}}_{nm}$, which will be necessary in eliminating critical measurement set (i.e., minimally dependent set). Furthermore, topological observability constraints can also be easily added because of their linearity with respect to the placement vector $\bdsb{\mathcal{V}}$ as in \cite{bei2005optimal,baldwin1993power,emami2010robust, chakrabarti2008optimal, chakrabarti2009placement}. We omit the full formulation due to lack of space.
\end{rmk}

\vspace{-0.3cm}
\subsection{Semi-Definite Programming (SDP) and Relaxation}
The eigenvalue problem in \eqref{PMU_opt_relaxed_critical_meas} can be reformulated via linear matrix inequalities using two dummy variables $\phi$ and $\beta$
\begin{align}\label{PMU_opt_relaxed}
	\underset{\bdsb{\mathcal{V}},\beta,\phi}{\max} ~~ \frac{\beta}{\phi}, ~~&\mathrm{s.t.} ~
    	~\mathcal{P}(\bdsb{\mathcal{V}}) +\mathcal{S}(\mathbf{v}_{\textrm{\tiny prior}}\mathbf{v}_{\textrm{\tiny prior}}^T,\bdsb{\mathcal{I}},\bdsb{\mathcal{F}}) \succeq \beta \mathbf{I}\\
	&
	\begin{bmatrix}
		\phi\mathbf{I}_{2N} & \mathbf{M}^{\frac{1}{2}}(\mathbf{I}_{2N}-\mathbf{J}_{\mathcal{V}})\nonumber\\
		(\mathbf{I}_{2N}-\mathbf{J}_{\mathcal{V}})^T\mathbf{M}^{\frac{T}{2}} & \mathbf{I}_{2N}
	\end{bmatrix}
	\succeq \mathbf{0}\nonumber\\	
	&~\mathcal{P}(\bdsb{\mathcal{V}}) +\mathcal{S}(\mathbf{v}_{\textrm{\tiny prior}}\mathbf{v}_{\textrm{\tiny prior}}^T,\bdsb{\mathcal{J}}_n,\bdsb{\mathcal{F}}) \succeq \beta_{\min}\mathbf{I}_{2N},~\forall n\in\mathcal{N}\nonumber\\
	&~\mathcal{P}(\bdsb{\mathcal{V}}) +\mathcal{S}(\mathbf{v}_{\textrm{\tiny prior}}\mathbf{v}_{\textrm{\tiny prior}}^T,\bdsb{\mathcal{I}},\bdsb{\mathcal{G}}_{nm}) \succeq \beta_{\min}\mathbf{I}_{2N},~\forall \{n,m\}\in\mathcal{E}\nonumber\\	
	&~
	\mathbf{J}_{\mathcal{V}} = \mathbf{I}_2\otimes \mathrm{diag}(\bdsb{\mathcal{V}}),~\mathcal{V}_n\in\{0,1\}\nonumber\\
      &~
    \mathbf{1}_N^T\bdsb{\mathcal{V}} \leq N_{\textrm{\tiny PMU}},~\bdsb{{{c}}}^T\bdsb{\mathcal{V}} \leq C_{\textrm{\tiny PMU}}.\nonumber
\end{align}



%
To avoid solving this complicated eigenvalue problem with integer constraints, we relax \eqref{PMU_opt_relaxed} by converting the integer constraint $\mathcal{V}_n\in\{0,1\}$ to a convex constraint $0\leq \mathcal{V}_n\leq 1$. Then, the optimization becomes a quasi-convex problem that needs to be solved in an iterative fashion via the classical bisection method by performing a sequence of semi-definite programs (SDP) feasibility problems \cite{boyd2004convex}. Clearly this consumes considerable computations and less desirable. Fortunately, since the objective \eqref{PMU_opt_relaxed} is a linear fractional function, the {\it Charnes-Cooper} transformation \cite{charnes1962programming} can be used to re-formulate the problem in \eqref{PMU_opt_relaxed} as a convex SDP, whose global optimum can be obtained in one pass.

\begin{proposition}
By letting $\gamma=1/\phi$, $\tau = \beta/\phi$ and $\bdsb{\xi} = \bdsb{\mathcal{V}}/\phi$, the global optimum solution to \eqref{PMU_opt_relaxed} can be determined by
\begin{align}\label{SDP_radius}
	\underset{\bdsb{\xi},\tau,\gamma}{\max} ~~ \tau,
	~~~~
	\mathrm{s.t.} &~\mathcal{P}(\bdsb{\xi}) +
    \gamma\mathcal{S}(\mathbf{v}_{\textrm{\tiny prior}}\mathbf{v}_{\textrm{\tiny prior}}^T,\bdsb{\mathcal{I}},\bdsb{\mathcal{F}}) \succeq \tau \mathbf{I}\\
	&~
	\begin{bmatrix}
		\mathbf{I} & \mathbf{M}^{\frac{1}{2}}(\gamma\mathbf{I}-\mathbf{I}_{\bdsb{\xi}})\nonumber\\
		(\gamma\mathbf{I}-\mathbf{I}_{\bdsb{\xi}})^T\mathbf{M}^{\frac{T}{2}} & \gamma\mathbf{I}
	\end{bmatrix}
	\succeq \mathbf{0}\nonumber\\
	&~\mathcal{P}(\bdsb{\xi}) + \gamma\mathcal{S}(\mathbf{v}_{\textrm{\tiny prior}}\mathbf{v}_{\textrm{\tiny prior}}^T,\bdsb{\mathcal{J}}_n,\bdsb{\mathcal{F}}) \succ \mathbf{0},~\forall n\nonumber\\
       &~\mathcal{P}(\bdsb{\xi}) + \gamma\mathcal{S}(\mathbf{v}_{\textrm{\tiny prior}}\mathbf{v}_{\textrm{\tiny prior}}^T,\bdsb{\mathcal{I}},\bdsb{\mathcal{G}}_{nm})\succ \mathbf{0},~\forall \{n,m\}\nonumber\\
	&~\mathbf{I}_{\bdsb{\xi}} = \mathbf{I}_2\otimes\mathrm{diag}\left(\bdsb{\xi}\right), \xi_n\in[0,\gamma]\nonumber\\
       &~\mathbf{1}_N^T\bdsb{\xi} \leq N_{\textrm{\tiny PMU}}\gamma, ~\bdsb{{{c}}}^T\bdsb{\xi} \leq \gamma C_{\textrm{\tiny PMU}},\nonumber
\end{align}
whose solution is mapped to the solution of \eqref{PMU_opt_relaxed} as 
\begin{align}\label{optimal_SDP}
	\bdsb{\mathcal{V}}^\star = \bdsb{\xi}^\star/\gamma^\star.
\end{align}
\end{proposition}
The solution $\bdsb{\mathcal{V}}^\star$ has real values but not the original binary values. Here we use a randomization technique \cite{luo2010semidefinite} to choose the solution by drawing a group of $\ell_\star$ binary vectors $\bdsb{\mathcal{V}}_\ell$ from a Bernoulli distribution on each entry with probabilities obtained from the solution $\bdsb{\mathcal{V}}^\star$. Then we compare the COP metric evaluated at the group of candidates $\{\bdsb{\mathcal{V}}_\ell\}_{\ell=1}^{\ell_\star}$ and choose the one that has the maximum as the optimal placement vector. This scheme approximates closely to the optimal solution of the original integer problem as shown in simulations.


\begin{figure}
\centering
\includegraphics[width=\linewidth,height=6cm]{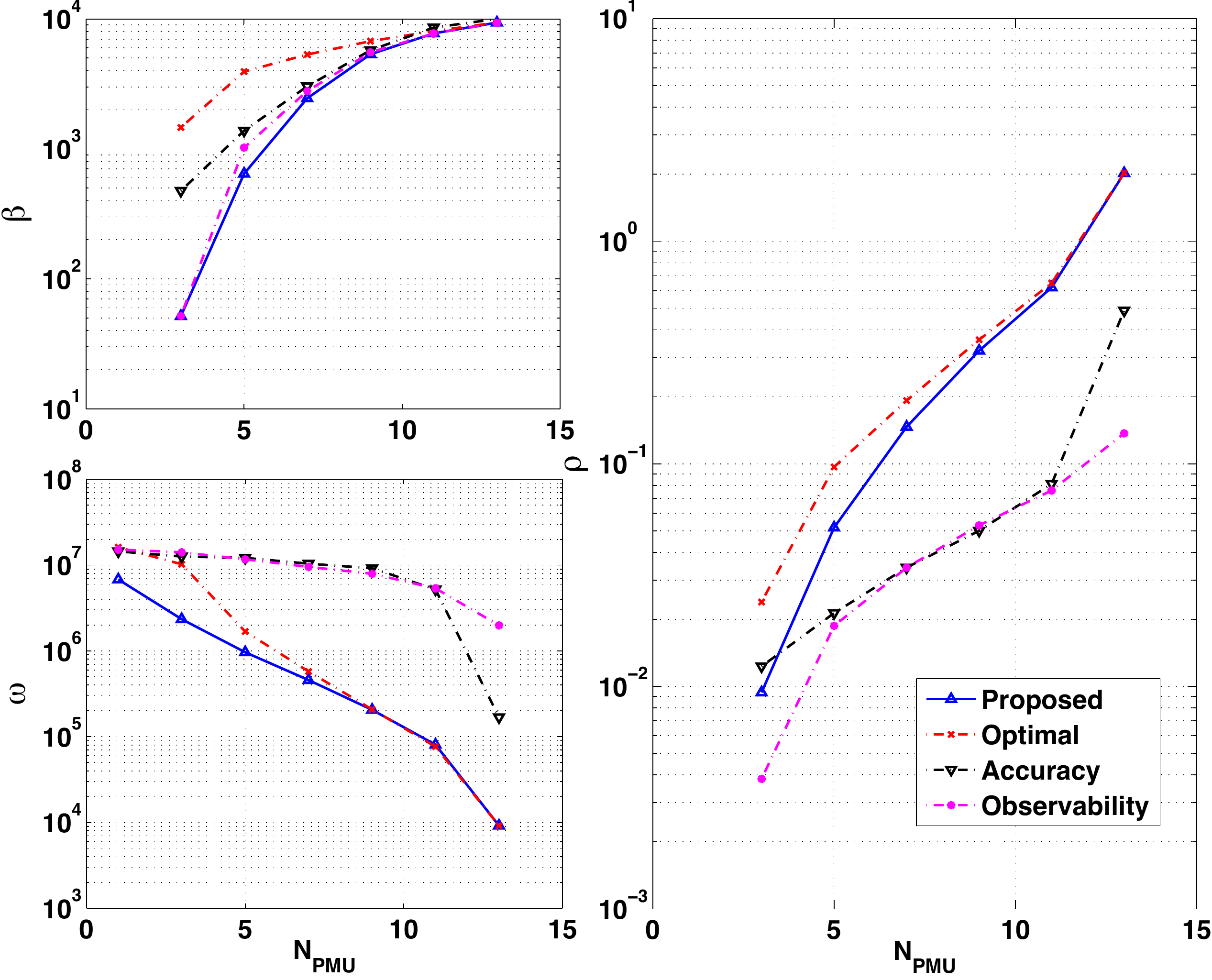}
\caption{Comparison of $\rho$,  $\phi$ and $\beta$ for the IEEE-14 system.}\label{fig.optimality}
\vspace{-0.4cm}
\end{figure}

\vspace{-0.2cm}
\section{Simulations}\label{simulations}
\vspace{-0.1cm}
In this section, we compare our {\it proposed} design in different systems mainly against the {\it accuracy} placement that optimizes estimation accuracy (i.e., $E$-{\it optimal} in \cite{li2011phasor,kekatos2011convex}) and an {\it observability} placement that satisfies system observability \cite{baldwin1993power} jointly with SCADA measurements. The measurements are generated with independent errors $\mathbf{R} = \sigma^2 \mathbf{I}$ and $\sigma^2=10^{-4}$. We demonstrate the optimality of our formulation in the IEEE-14 system, and extend the comparison on the convergence and estimation performance for IEEE 30 and 118 systems, using $15\%$ of all SCADA measurements provided at random\footnote{The number of SCADA measurements in each experiment is $15$\%$\times(4N+8L)$, where $N$ is the number of buses and $L$ is the number of lines.}. 


\vspace{-0.3cm}
\subsection{IEEE 14 bus System {(Fig. \ref{fig.optimality} and Fig. \ref{fig.sequence})}}
We show the optimality of the {\it proposed} placement in Fig. \ref{fig.optimality} by comparing $\rho$, $\beta$ and $\phi$ against the {\it accuracy}, the {\it observability} and most importantly the exact {\it optimal} PMU placement in the IEEE-14 system for $N_{\textrm{\tiny PMU}}=1,\cdots,13$, where the exact {\it optimal} solution is obtained by an exhaustive search in the non-relaxed problem \eqref{PMU_opt_relaxed_critical_meas}. It is seen from $\beta$ that under $15$\% SCADA measurements, the system remains unobservable until $N_{\textrm{\tiny PMU}}=3$ since $\beta=0$ is not shown on the curve.

\begin{figure}
\centering
\includegraphics[width=0.8\linewidth]{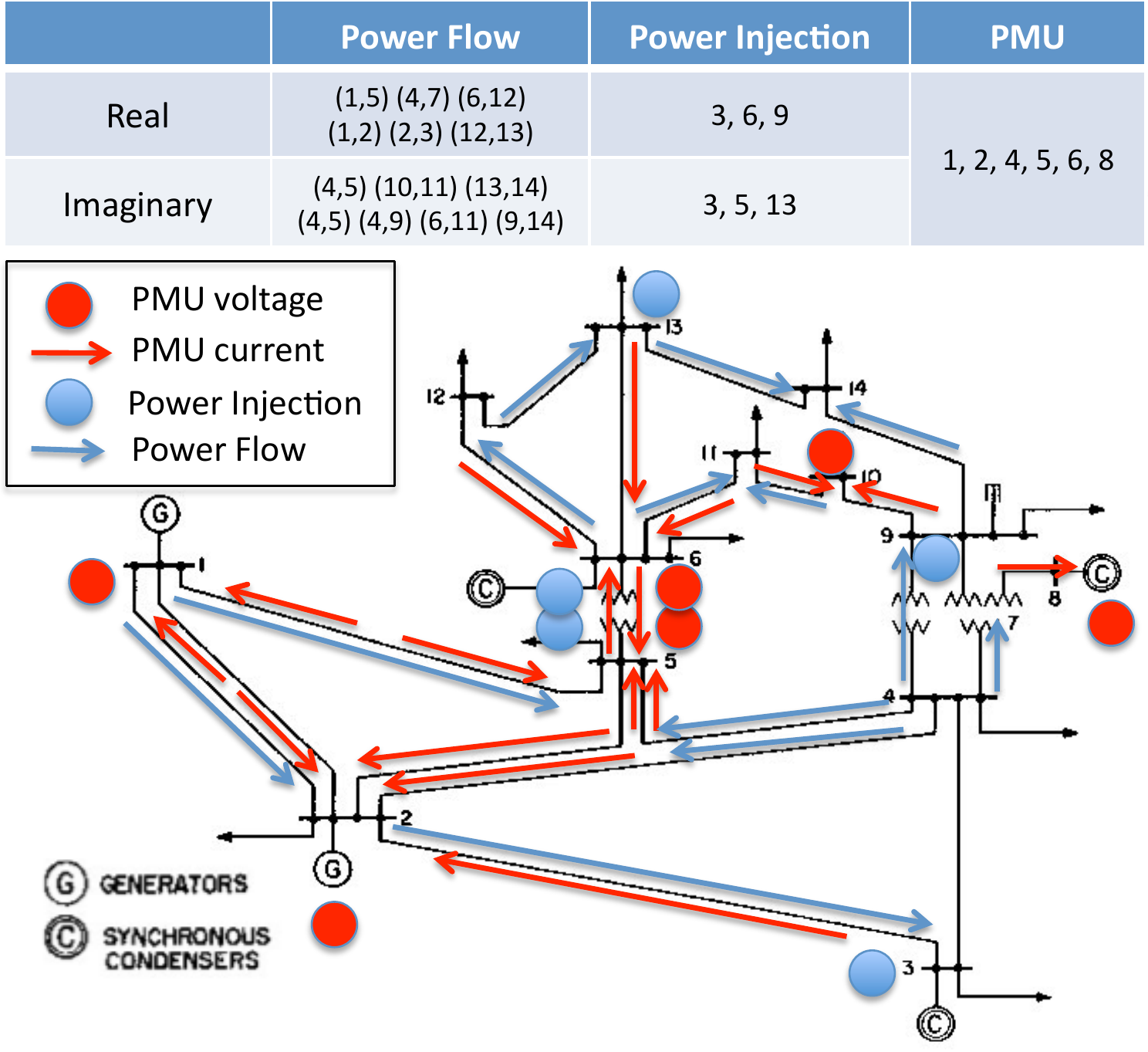}
\caption{Example of the proposed placement for the IEEE-14 system.}\label{fig.sequence}
\vspace{-0.4cm}
\end{figure}

A significant gap can be seen in Fig. \ref{fig.optimality} between the {\it proposed}, the {\it optimal} and the {\it accuracy} schemes. It is clear that the {\it proposed} scheme gives a uniformly greater $\rho$ than the {\it accuracy} scheme, and closely touches the {\it optimal} solution. Clearly, the {\it accuracy} design achieves a larger $\beta$ than the {\it proposed} scheme, but this quantity is less sensitive to the PMU placement than $\phi$ for all $N_{\textrm{\tiny PMU}}$. This implies that the estimation accuracy of the hybrid state estimation is not very sensitive to the placement, because of the presence of SCADA measurements. In fact, convergence is a more critical issue. In particular, when the PMU budget is low (i.e. $N_{\textrm{\tiny PMU}}$ is small), the {\it accuracy} does not provide discernible improvement on $\phi$ (thus $\rho$) while the {\it optimal} and {\it proposed} schemes considerably lower $\phi$ and increase $\rho$, which stabilizes and accelerates the algorithm convergence without affecting greatly accuracy.

In Fig. \ref{fig.sequence}, we show an example of the {\it proposed} placement with $N_{\textrm{\tiny PMU}}=6$ in one experiment where there are $19$ SCADA measurements ($15\%$ of total) marked in ``blue'' while there are PMU measurements marked in ``red''. It can be seen that the system is always observable even with single failure because each node is metered by the measurements at least twice so there is enough redundancy to avoid critical measurements.

\begin{figure}[!t]
\begin{center}
{\subfigure[][IEEE-30 bus system]{\resizebox{0.45\textwidth}{!}{\includegraphics{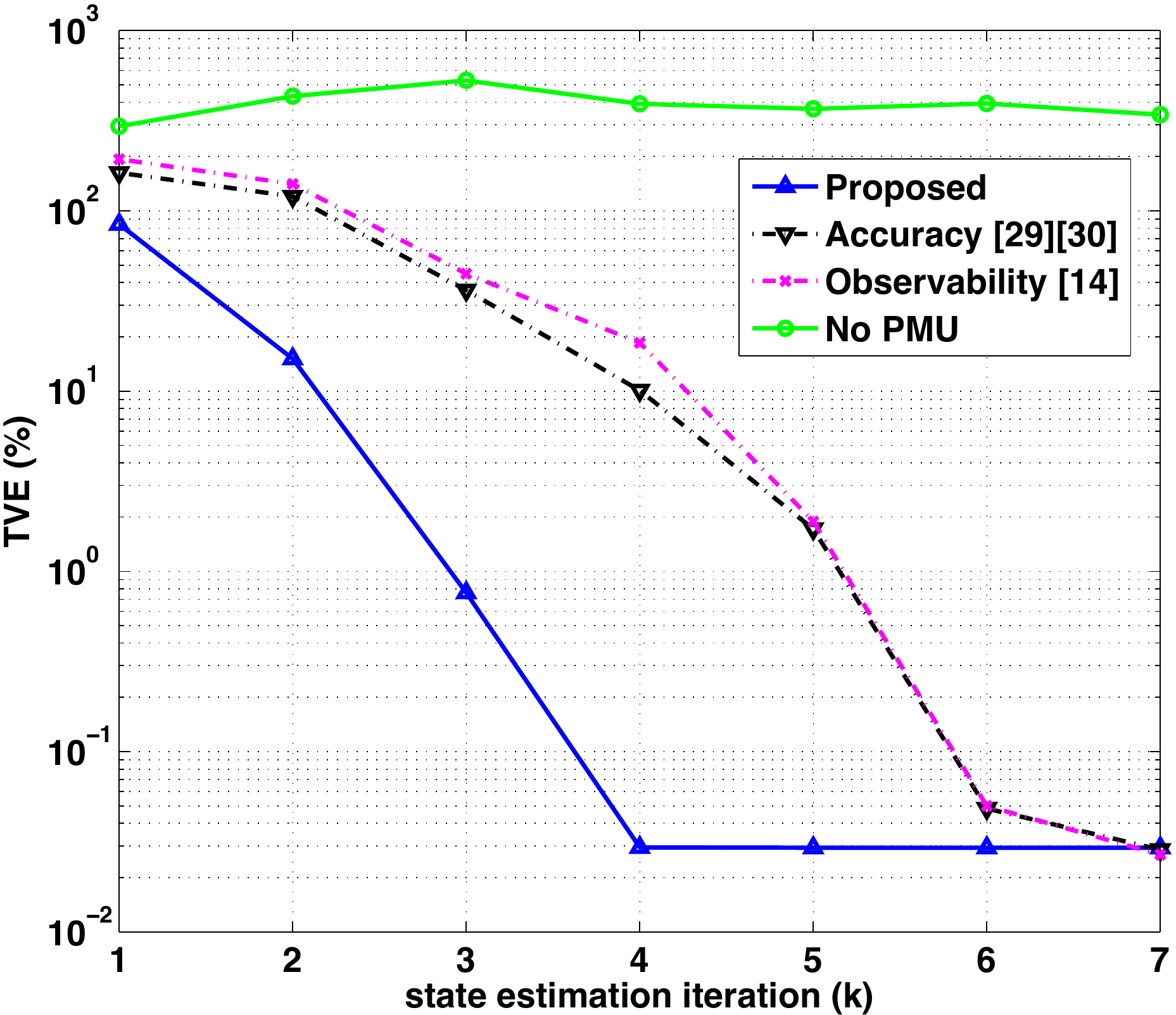}\label{fig.estimation_accuracy_30}}}}\\
{\subfigure[][IEEE-118 bus system]{\resizebox{0.45\textwidth}{!}{\includegraphics{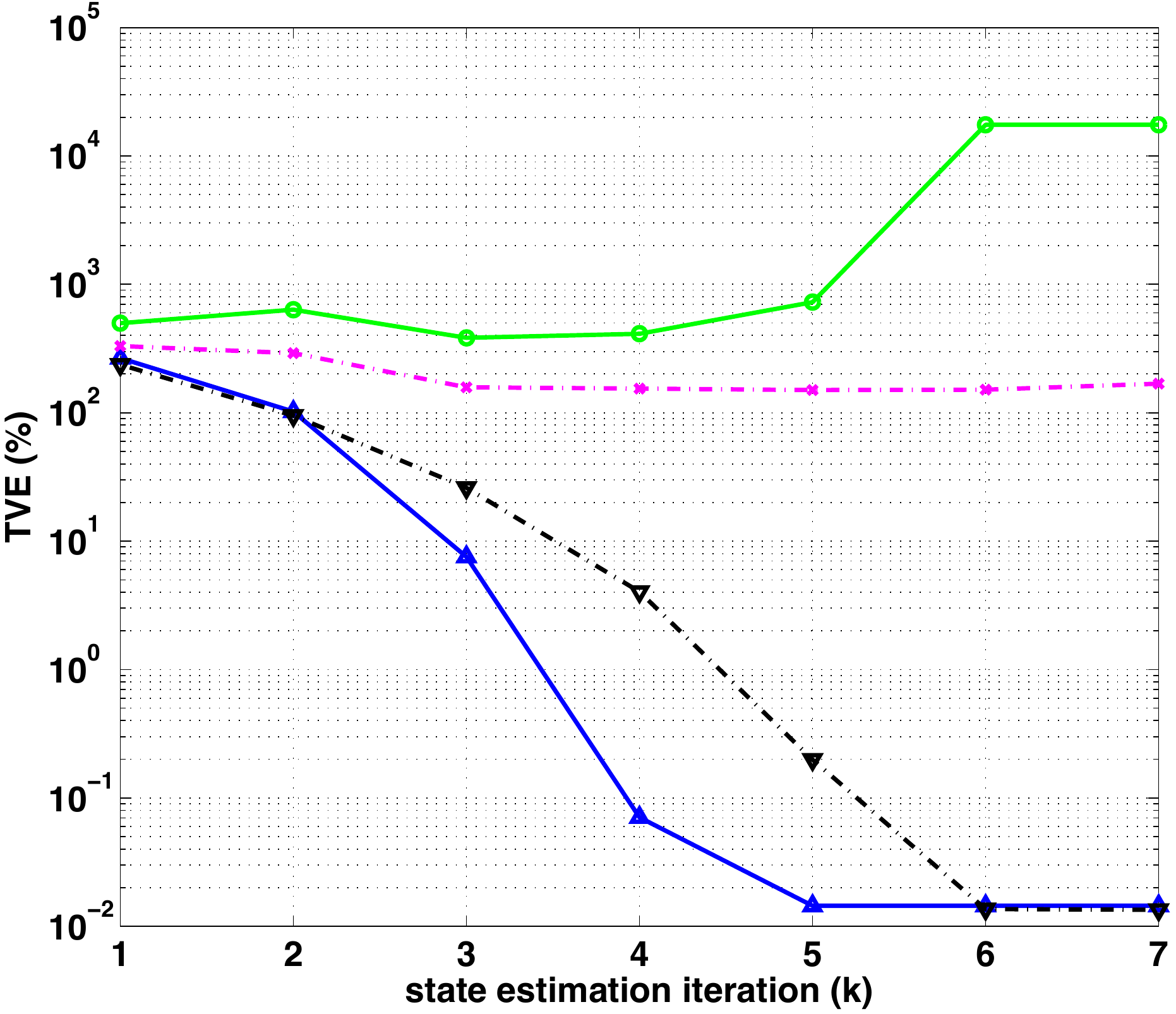}\label{fig.estimation_accuracy_118}}}}
\end{center}\vspace{-0.4cm}
\caption{TVE curves for the IEEE-30 and 118 system.}
\vspace{-0.5cm}
\end{figure}

\vspace{-0.3cm}
\subsection{IEEE 30-bus and 118-bus Systems}

We illustrate the estimation convergence and performance of our {\it proposed} placement against the alternatives above and the case with no PMUs, in terms of the total vector error (TVE) in \cite{martin2008exploring} for evaluating the accuracy of PMU-related state estimates
\begin{align}
	\mathrm{TVE}_k =  \left\|\mathbf{v}^k-\mathbf{v}_{\textrm{\tiny true}}\right\|/\left\|\mathbf{v}_{\textrm{\tiny true}}\right\| \times 100\%
\end{align}	
for each iteration $k$. This shows the decrease of TVE as the Gauss-Newton proceeds iteratively, which is a typical way to illustrate convergence behavior and the asymptotic accuracy upon convergence. With $17\%$ PMU deployment, we compare the TVE curves for the IEEE-30 system with $N_{\textrm{\tiny PMU}}=5$ in Fig. \ref{fig.estimation_accuracy_30} and the IEEE-118 system with $N_{\textrm{\tiny PMU}}=20$ in Fig. \ref{fig.estimation_accuracy_118}. To verify the robustness to initialization (numerical stability) and the convergence rate, the TVE curves are averaged over $200$ experiments. For each experiment, we generate a placement guaranteeing observability for the {\it observability} placement according to \cite{baldwin1993power}, and use a non-informative initializer $\mathbf{v}_{\textrm{\tiny prior}} = [\mathbf{1}_N^T+0.1\bdsb{\varepsilon}^T,\mathbf{0}_N^T]^T$ perturbed by a zero mean Gaussian error vector $\bdsb{\varepsilon}$ with $\mathbb{E}\left[\bdsb{\varepsilon}\bdsb{\varepsilon}^T\right]=\mathbf{I}$. We leave the imaginary part unperturbed because phases are usually small.

It is seen in Fig. \ref{fig.estimation_accuracy_30} that if there are no PMU installed, it is possible that the algorithm does not converge while the {\it proposed} placement scheme converges stably. The performance of the {\it observability} placement is not stably guaranteed even if it satisfies observability because it diverges under perturbations for the 118 system in Fig. \ref{fig.estimation_accuracy_30}. A similar divergent trend can be observed if the initialization is very inaccurate, regardless of how it is set. Consistent with Theorem \ref{cor_convergence}, since the noise $\sigma^2$ is small, the algorithm converges quadratically for the {\it proposed} and the {\it accuracy} placement, but the convergence rates vary greatly. Although the asymptotic TVE remains comparable, the {\it proposed} placement considerably accelerates the convergence compared to the {\it observability} and {\it accuracy} placement. 

\vspace{-0.2cm}
\section{Conclusions}
\vspace{-0.1cm}
In this paper, we propose a useful metric, referred to as COP, to evaluate the convergence and accuracy of hybrid PSSE for a given sensor deployment, where PMUs are used to initialize the Gauss-Newton iterative estimation. The COP metric is derived from the convergence analysis of the Gauss-Newton state estimation procedures, which is a joint measure for convergence $\phi(\bdsb{\mathcal{V}})$ and the FIM as a measure for accuracy and observability $\beta(\bdsb{\mathcal{V}})$. We optimize our placement strategy by maximizing the COP metric $\beta(\bdsb{\mathcal{V}})/\phi(\bdsb{\mathcal{V}})$ via a simple SDP, and the critical measurement constraints in the SDP formulation further ensure that the numerical observability $\beta(\bdsb{\mathcal{V}})$ is bounded away from zero up to a tolerable point. Finally, the simulations confirm numerically the convergence and estimation performance of the proposed scheme.

\vspace{-0.2cm}
\appendices
\section{Power Flow Equations and Jacobian Matrix}\label{power_flow_equations_appendix}
The admittance matrix $\mathbf{Y}=[-Y_{nm}]_{N\times N}$, includes line admittances $Y_{nm}=G_{nm}+\mathrm{i}B_{nm}$, $\{n,m\}\in\mathcal{E}$ and bus admittance-to-ground $\bar{Y}_{nm}=\bar{G}_{nm}+\mathrm{i}\bar{B}_{nm}$ in the $\Pi$-model of line $\{n,m\}\in\mathcal{E}$, and self-admittance $Y_{nn}= - \sum_{m\neq n} (\bar{Y}_{nm} + Y_{nm})$. Using the canonical basis $\mathbf{e}_n=[0,\cdots,1,\cdots,0]^T$ and the matrix $\mathbf{Y}$, we define the following matrices
\begin{align*}
	\mathbf{Y}_n &\triangleq \mathbf{e}_n\mathbf{e}_n^T\mathbf{Y},
	\quad
	\mathbf{Y}_{nm} &\triangleq (Y_{nm}+\bar{Y}_{nm})\mathbf{e}_n\mathbf{e}_n^T - Y_{nm}\mathbf{e}_n\mathbf{e}_m^T.
\end{align*}
Letting $\mathbf{G}_n = \Re\{\mathbf{Y}_n\}$, $\mathbf{B}_n=\Im\{\mathbf{Y}_n\}$, $\mathbf{G}_{nm}=\Re\{\mathbf{Y}_{nm}\}$ and $\mathbf{B}_{nm}=\Im\{\mathbf{Y}_{nm}\}$, we define the following matrices
\begin{align*}
    \mathbf{N}_{P,n}
    &\triangleq
    \begin{bmatrix}
        \mathbf{G}_n & - \mathbf{B}_n \\
        \mathbf{B}_n & \mathbf{G}_n
    \end{bmatrix}~~~~~~
    \mathbf{N}_{Q,n}
    \triangleq
    -
    \begin{bmatrix}
        \mathbf{B}_n & \mathbf{G}_n\\
        -\mathbf{G}_n & \mathbf{B}_n
    \end{bmatrix}\\
    \mathbf{E}_{P,nm}
    &\triangleq
    \begin{bmatrix}
        \mathbf{G}_{nm} &  - \mathbf{B}_{nm} \\
        \mathbf{B}_{nm}  & \mathbf{G}_{nm}
    \end{bmatrix}~
    \mathbf{E}_{Q,nm}
    \triangleq
    -
    \begin{bmatrix}
        \mathbf{B}_{nm}  & \mathbf{G}_{nm}\\
        -\mathbf{G}_{nm} & \mathbf{B}_{nm}
    \end{bmatrix}\\
    \mathbf{C}_{I,nm}
    &\triangleq
    \begin{bmatrix}
        \mathbf{G}_{nm}  & \mathbf{0}\\
        \mathbf{0}  & -\mathbf{B}_{nm}
    \end{bmatrix}~~
    \mathbf{C}_{J,nm}
    \triangleq
    ~
    \begin{bmatrix}
        \mathbf{B}_{nm}  & \mathbf{0}\\
        \mathbf{0}  & \mathbf{G}_{nm}
    \end{bmatrix}.
\end{align*}
The SCADA system collects active/reactive injection $(P_n,Q_n)$ at bus $n$ and flow $(P_{nm},Q_{nm})$ at bus $n$ on line $\{n,m\}$
		\begin{align}
			P_n
			&=  \mathbf{v}^T\mathbf{N}_{P,n}\mathbf{v},\quad\quad~
			Q_n
			=  \mathbf{v}^T\mathbf{N}_{Q,n}\mathbf{v},\label{power_flow_eq_SCADA_power_inj}\\
			P_{nm}
			&= \mathbf{v}^T\mathbf{E}_{P,nm}\mathbf{v},\quad
			Q_{nm}
			= \mathbf{v}^T\mathbf{E}_{Q,nm}\mathbf{v},\label{power_flow_eq_SCADA_power_flow}
		\end{align}
and stack them in the power flow equations
		\begin{align}
			\mathbf{f}_{\mathcal{I}}(\mathbf{v}) &= [\cdots,P_n, \cdots, \cdots, Q_n, \cdots]^T\\
			\mathbf{f}_{\mathcal{F}}(\mathbf{v}) &= [\cdots,P_{nm},\cdots,\cdots,Q_{nm}, \cdots]^T.
		\end{align}
The WAMS collects the voltage $(\Re\{V_n\},\Im\{V_n\})$ at bus $n$ and the current $(I_{nm},J_{nm})$ on line $\{n,m\}$ measured at bus $n$
		\begin{align}\label{power_flow_eq_PMU}
			I_{nm}
			&= \left(\mathbf{1}_2\otimes\mathbf{e}_n\right)^T\mathbf{C}_{I,nm}\mathbf{v}\\
			J_{nm}
			&= \left(\mathbf{1}_2\otimes\mathbf{e}_n\right)^T\mathbf{C}_{J,nm}\mathbf{v},
		\end{align}
where $\otimes$ is the Kronecker product, and stacks them as
\begin{align}
	\mathbf{f}_{\mathcal{V}}(\mathbf{v}) = \mathbf{v},~\mathbf{f}_{\mathcal{C}}(\mathbf{v}) = [\cdots,I_{nm},\cdots,\cdots,J_{nm}, \cdots]^T.
\end{align}
The Jacobian $\mathbf{F}(\mathbf{v})$ can be derived from \eqref{power_flow_eq_PMU}, \eqref{power_flow_eq_SCADA_power_inj} and \eqref{power_flow_eq_SCADA_power_flow}
\begin{align}\label{Jacobian}
    \mathbf{F}^T(\mathbf{v})
    &=
        \begin{bmatrix}
    		\mathbf{I}_{2N},~
		\mathbf{H}^T_{\mathcal{C}},~
		\mathbf{H}^T_{\mathcal{I}}(\mathbf{I}_{2N}\otimes\mathbf{v}),~\mathbf{H}_{\mathcal{F}}^T
		(\mathbf{I}_{4L}\otimes\mathbf{v})    \end{bmatrix}
\end{align}
where
\begin{align*}
	\mathbf{H}_{\mathcal{C}} &\triangleq [\cdots,\mathbf{H}_{I,n}^T,\cdots,\cdots,\mathbf{H}_{J,n}^T,\cdots]^T\\
	\mathbf{H}_{\mathcal{I}} &\triangleq [ \cdots,\mathbf{N}_{P,n}+\mathbf{N}_{P,n}^T,\cdots,\mathbf{N}_{Q,n}+\mathbf{N}_{Q,n}^T,\cdots]^T\\
	\mathbf{H}_{\mathcal{F}} &\triangleq[\cdots,\mathbf{E}_{P,nm}+\mathbf{E}_{P,nm}^T,\cdots,\mathbf{E}_{Q,nm}+\mathbf{E}_{Q,nm}^T,\cdots]^T
\end{align*}
using $\mathbf{S}_n \triangleq \mathbf{I}_{L_n}\otimes\left(\mathbf{1}_2\otimes\mathbf{e}_n\right)^T$ and
\begin{align}
	\mathbf{H}_{I,n} &\triangleq \mathbf{S}_n \mathbf{C}_{I,n},\quad \mathbf{C}_{I,n} \triangleq [\cdots,\mathbf{C}_{I,nm}^T,\cdots]^T\\
	\mathbf{H}_{J,n} &\triangleq \mathbf{S}_n \mathbf{C}_{J,n},\quad \mathbf{C}_{J,n} \triangleq [\cdots,\mathbf{C}_{J,nm}^T,\cdots]^T.\nonumber
\end{align}

\begin{figure*}[t]
\begin{align}\label{full_exp}
   \left\|\mathbf{F}_{\mathcal{A}}(\mathbf{v})-\mathbf{F}_{\mathcal{A}}(\mathbf{v}')\right\|_F^2
   &=\mathrm{Tr}\left[\sum_n \frac{\mathcal{I}_n}{\sigma_{\mathcal{I}}^2}(\mathbf{N}_{P,n}+\mathbf{N}_{P,n}^T)^T(\mathbf{v}-\mathbf{v}')(\mathbf{v}-\mathbf{v}')^T(\mathbf{N}_{P,n}+\mathbf{N}_{P,n}^T)\right]\nonumber\\
   &+\mathrm{Tr}\left[\sum_n \frac{\mathcal{I}_n}{\sigma_{\mathcal{I}}^2}(\mathbf{N}_{Q,n}+\mathbf{N}_{Q,n}^T)^T(\mathbf{v}-\mathbf{v}')(\mathbf{v}-\mathbf{v}')^T(\mathbf{N}_{Q,n}+\mathbf{N}_{Q,n}^T)\right]\nonumber\\
    &+\mathrm{Tr}\left[\sum_{n,m} \frac{\mathcal{F}_{nm}}{\sigma_{\mathcal{F}}^2}(\mathbf{E}_{P,nm}+\mathbf{E}_{P,nm}^T)^T(\mathbf{v}-\mathbf{v}')(\mathbf{v}-\mathbf{v}')^T(\mathbf{E}_{P,nm}+\mathbf{E}_{P,nm}^T)\right]\nonumber\\
    &+\mathrm{Tr}\left[\sum_{n,m} \frac{\mathcal{F}_{nm}}{\sigma_{\mathcal{F}}^2}(\mathbf{E}_{Q,nm}+\mathbf{E}_{Q,nm}^T)^T(\mathbf{v}-\mathbf{v}')(\mathbf{v}-\mathbf{v}')^T(\mathbf{E}_{Q,nm}+\mathbf{E}_{Q,nm}^T)\right]
\end{align}
\end{figure*}

\section{Proof of Lemma \ref{lem_error_recursion}}\label{proof_lem_error_recursion}
\subsection{Proof of Jacobian properties}
We first prove that $\mathbf{F}_{\mathcal{A}}(\mathbf{v})$ satisfies the following
\begin{align}\label{lipschitz_F}
	 \left\|\mathbf{F}_{\mathcal{A}}(\mathbf{v})-\mathbf{F}_{\mathcal{A}}(\mathbf{v}')\right\|_F^2
   &=\mathrm{Tr}\left[(\mathbf{v}-\mathbf{v}')(\mathbf{v}-\mathbf{v}')^T\mathbf{M}\right]\\
   &= (\mathbf{v}-\mathbf{v}')^T\mathbf{M}(\mathbf{v}-\mathbf{v}'),
\end{align}
where $\mathbf{M}$ is a constant matrix defined in \eqref{M} once the SCADA measurement placements are fixed.
From the expression of the Jacobian in the Appendix of the manuscript, we have
\begin{align}
	\mathbf{F}_{\mathcal{A}}(\mathbf{v})-\mathbf{F}_{\mathcal{A}}(\mathbf{v}')
	=
	\begin{bmatrix}
		\mathbf{0}_{2N\times 2N}\\
		\mathbf{0}_{4L\times 2N}\\
		 \left[\mathbf{R}_{\mathcal{I}}^{-\frac{1}{2}} \mathbf{J}_\mathcal{I}\otimes(\mathbf{v}-\mathbf{v}')^T\right]\mathbf{H}_{\mathcal{I}}\\
		 \left[\mathbf{R}_{\mathcal{F}}^{-\frac{1}{2}} \mathbf{J}_\mathcal{F}\otimes(\mathbf{v}-\mathbf{v}')^T\right]\mathbf{H}_{\mathcal{F}}
	\end{bmatrix}.
\end{align}
Using the $F$-norm definition $\left\|\cdot\right\|_F^2=\mathrm{Tr}\left[(\cdot)^T(\cdot)\right]$ together with the property of Kronecker products, we have
\begin{align*}
    &\left\|\mathbf{F}_{\mathcal{A}}(\mathbf{v})-\mathbf{F}_{\mathcal{A}}(\mathbf{v}')\right\|_F^2\\
    &= \mathrm{Tr}\left[\mathbf{H}_{\mathcal{I}}^T\left(\mathbf{J}_\mathcal{I}\mathbf{R}_{\mathcal{I}}^{-1}\mathbf{J}_\mathcal{I}^T\otimes(\mathbf{v}-\mathbf{v}')(\mathbf{v}-\mathbf{v}')^T\right)\mathbf{H}_{\mathcal{I}}\right]\\
    &~~~+\mathrm{Tr}\left[\mathbf{H}_{\mathcal{F}}^T\left(\mathbf{J}_\mathcal{F}\mathbf{R}_{\mathcal{F}}^{-1}\mathbf{J}_\mathcal{F}^T\otimes(\mathbf{v}-\mathbf{v}')(\mathbf{v}-\mathbf{v}')^T\right)\mathbf{H}_{\mathcal{F}}\right].
\end{align*}
Using the property of trace operators, the sub-matrices of $\mathbf{H}_{\mathcal{I}}$ and $\mathbf{H}_{\mathcal{F}}$, we can express the above norm by expanding the Kronecker product $\otimes$ and re-arrange the summation as \eqref{full_exp}, which leads to the result in \eqref{lipschitz_F}
with the matrix $\mathbf{M}$
\begin{align}\label{M}
	\mathbf{M}
	=& \sum_{n=1}^N \frac{\mathcal{I}_n}{\sigma_{\mathcal{I}}^2}\left(\mathbf{N}_{P,n}+\mathbf{N}_{P,n}^T\right)^T\left(\mathbf{N}_{P,n}+\mathbf{N}_{P,n}^T\right)\\
	&+ \sum_{n=1}^N \frac{\mathcal{I}_n}{\sigma_{\mathcal{I}}^2}\left(\mathbf{N}_{Q,n}+\mathbf{N}_{Q,n}^T\right)^T\left(\mathbf{N}_{Q,n}+\mathbf{N}_{Q,n}^T\right) \nonumber \\
	&+ \sum_{n,l} \frac{\mathcal{F}_{nm}}{\sigma_{\mathcal{F}}^2}\left(\mathbf{E}_{P,nm}+\mathbf{E}_{P,nm}^T\right)^T\left(\mathbf{E}_{P,nm}+\mathbf{E}_{P,nm}^T\right)\nonumber\\
 	&+\sum_{n,l} \frac{\mathcal{F}_{nm}}{\sigma_{\mathcal{F}}^2}\left(\mathbf{E}_{Q,nm}+\mathbf{E}_{Q,nm}^T\right)^T\left(\mathbf{E}_{Q,nm}+\mathbf{E}_{Q,nm}^T\right)\nonumber.
\end{align}	

\subsection{Proof of Error Recursion}
Using this result we are now ready to prove Lemma 1.
With the results we proved in \cite[Lem. 2]{li2012convergence}, we have the iterative error expressed as
\begin{align}\label{orig.recursion}
	\mathbf{v}^{k+1}-\mathbf{v}_{\textrm{\tiny est}} 
       &= \mathbf{F}_{\mathcal{A}}^\dagger(\mathbf{v}^k)\left[\mathbf{F}_{\mathcal{A}}(\mathbf{v}^k)\left(\mathbf{v}^k-\mathbf{v}_{\textrm{\tiny est}}\right) +\mathbf{f}_{\mathcal{A}}(\mathbf{v}_{\textrm{\tiny est}}) -\mathbf{f}_{\mathcal{A}}(\mathbf{v}^k)\right]\nonumber\\
	&~~~
    +\left[\mathbf{F}_{\mathcal{A}}^\dagger(\mathbf{v}^k)-\mathbf{F}_{\mathcal{A}}^\dagger(\mathbf{v}_{\textrm{\tiny est}})\right]
      \Big[\mathbf{z}_{\mathcal{A}}-\mathbf{f}_{\mathcal{A}}(\mathbf{v}_{\textrm{\tiny est}})\Big],
\end{align}
whose norm can be bounded as
\begin{align}\label{orig.recursion.bounded}
	&\left\|\mathbf{v}^{k+1}-\mathbf{v}_{\textrm{\tiny est}}\right\| \\ 
       &\leq \left\|\mathbf{F}_{\mathcal{A}}^\dagger(\mathbf{v}^k)\right\|\left\|\left[\mathbf{F}_{\mathcal{A}}(\mathbf{v}^k)\left(\mathbf{v}^k-\mathbf{v}_{\textrm{\tiny est}}\right) +\mathbf{f}_{\mathcal{A}}(\mathbf{v}_{\textrm{\tiny est}}) -\mathbf{f}_{\mathcal{A}}(\mathbf{v}^k)\right]\right\|\nonumber\\
	&~~~
    +\left\|\left[\mathbf{F}_{\mathcal{A}}^\dagger(\mathbf{v}^k)-\mathbf{F}_{\mathcal{A}}^\dagger(\mathbf{v}_{\textrm{\tiny est}})\right]
      \Big[\mathbf{z}_{\mathcal{A}}-\mathbf{f}_{\mathcal{A}}(\mathbf{v}_{\textrm{\tiny est}})\Big]\right\|,
\end{align}

First of all, by the definition of matrix norm, we have
\begin{align}
	\left\|\mathbf{F}_{\mathcal{A}}^\dagger(\mathbf{v})\right\|^2 = \left\|\mathbf{F}_{\mathcal{A}}^\dagger(\mathbf{v})\left(\mathbf{F}_{\mathcal{A}}^\dagger(\mathbf{v}))\right)^T\right\|=\left\|\left(\mathbf{F}_{\mathcal{A}}^T(\mathbf{v})\mathbf{F}_{\mathcal{A}}(\mathbf{v})\right)^{-1}\right\|.
\end{align}	
Then evidently, the norm of a matrix inverse corresponds to the reciprocal of the eigenvalue of that matrix with the smallest magnitude. By the definition
\begin{align}\label{beta}
	\beta(\bdsb{\mathcal{V}}) &=  \underset{\mathbf{v}\in\mathbb{V}}{\inf}~\lambda_{\min}
	\left[\mathbf{F}_{\mathcal{A}}^T(\mathbf{v})\mathbf{F}_{\mathcal{A}}(\mathbf{v})\right],
\end{align}
this quantity is bounded as $\beta(\bdsb{\mathcal{V}})$, and therefore
\begin{align}\label{pseudo_inverse_bound}
	\left\|\mathbf{F}_{\mathcal{A}}^\dagger(\mathbf{v})\right\|\leq 1/\sqrt{\beta(\bdsb{\mathcal{V}})}.
\end{align}

\subsubsection{\bf The First Term}
The first term can be written with the mean-value theorem as
\begin{align}\label{lipshitz_step}
    &\mathbf{F}_{\mathcal{A}}(\mathbf{v}^k)\left(\mathbf{v}^k-\mathbf{v}_{\textrm{\tiny est}}\right) +\mathbf{f}_{\mathcal{A}}(\mathbf{v}_{\textrm{\tiny est}}) -\mathbf{f}_{\mathcal{A}}(\mathbf{v}^k)\\
    &=\int_{0}^{1} \left[\mathbf{F}_{\mathcal{A}}(\mathbf{v}^k)-\mathbf{F}_{\mathcal{A}}\left(\mathbf{v}_{\textrm{\tiny est}}+t(\mathbf{v}^k-\mathbf{v}_{\textrm{\tiny est}})\right)\right]\left(\mathbf{v}^k-\mathbf{v}_{\textrm{\tiny est}}\right)\mathrm{d}t\nonumber\\
    &=\frac{1}{2}\left[\mathbf{F}_{\mathcal{A}}(\mathbf{v}^k)-\mathbf{F}_{\mathcal{A}}(\mathbf{v}_{\textrm{\tiny est}})\right]\left(\mathbf{v}^k-\mathbf{v}_{\textrm{\tiny est}}\right),
\end{align}
where the second equality comes from the linearity of the Jacobian in \eqref{Jacobian} in terms of the argument. Then the first term in \eqref{orig.recursion} can be bounded as
\begin{align*}
	& \left\|\mathbf{F}_{\mathcal{A}}(\mathbf{v}^k)\left(\mathbf{v}^k-\mathbf{v}_{\textrm{\tiny est}}\right) +\mathbf{f}_{\mathcal{A}}(\mathbf{v}_{\textrm{\tiny est}}) -\mathbf{f}_{\mathcal{A}}(\mathbf{v}^k)\right\|\\
	&\leq \frac{1}{2}\left\|\mathbf{F}_{\mathcal{A}}(\mathbf{v}^k)-\mathbf{F}_{\mathcal{A}}(\mathbf{v}_{\textrm{\tiny est}})\right\|\left\|\mathbf{v}^k-\mathbf{v}_{\textrm{\tiny est}}\right\|.
\end{align*}

\subsubsection{\bf The Second Term} From \cite[Lem. 1]{salzo2011convergence}, the second term in \eqref{orig.recursion} is bounded as
\begin{align*}
	&\left\|\left[\mathbf{F}_{\mathcal{A}}^\dagger(\mathbf{v}^k)-\mathbf{F}_{\mathcal{A}}^\dagger(\mathbf{v}_{\textrm{\tiny est}})\right]\left[\mathbf{z}_{\mathcal{A}}-\mathbf{f}_{\mathcal{A}}(\mathbf{v}_{\textrm{\tiny est}})\right]\right\|\\
	&\leq \sqrt{2}\left\|\mathbf{F}_{\mathcal{A}}^\dagger(\mathbf{v}^k)\right\|\left\|\mathbf{F}_{\mathcal{A}}^\dagger(\mathbf{v}_{\textrm{\tiny est}})\right\|\left\|\mathbf{F}_{\mathcal{A}}(\mathbf{v}^k)-\mathbf{F}_{\mathcal{A}}(\mathbf{v}_{\textrm{\tiny est}})\right\|\left\|\mathbf{z}_{\mathcal{A}}-\mathbf{f}_{\mathcal{A}}(\mathbf{v}_{\textrm{\tiny est}})\right\|.
\end{align*}
Using \eqref{pseudo_inverse_bound}, we have
\begin{align*}
	&\left\|\left[\mathbf{F}_{\mathcal{A}}^\dagger(\mathbf{v}^k)-\mathbf{F}_{\mathcal{A}}^\dagger(\mathbf{v}_{\textrm{\tiny est}})\right]\left[\mathbf{z}_{\mathcal{A}}-\mathbf{f}_{\mathcal{A}}(\mathbf{v}_{\textrm{\tiny est}})\right]\right\|\\
	&\leq \frac{\sqrt{2}}{\beta(\mathcal{V})}
	\left\|\mathbf{F}_{\mathcal{A}}(\mathbf{v}^k)-\mathbf{F}_{\mathcal{A}}(\mathbf{v}_{\textrm{\tiny est}})\right\|\left\|\mathbf{z}_{\mathcal{A}}-\mathbf{f}_{\mathcal{A}}(\mathbf{v}_{\textrm{\tiny est}})\right\|.
\end{align*}

Thus finally, the original recursion bounded in \eqref{orig.recursion.bounded} can be simplified as
\begin{align}\label{new.recursion}
	\left\|\mathbf{v}^{k+1}-\mathbf{v}_{\textrm{\tiny est}}\right\|
   	&\leq  \frac{\left\|\mathbf{v}^k-\mathbf{v}_{\textrm{\tiny est}}\right\|}{2\sqrt{\beta(\bdsb{\mathcal{V}})}}\left\|\mathbf{F}_{\mathcal{A}}(\mathbf{v})-\mathbf{F}_{\mathcal{A}}(\mathbf{v}')\right\|\\
	&~~+ \frac{\sqrt{2}\|\mathbf{z}_{\mathcal{A}}-\mathbf{f}_{\mathcal{A}}(\mathbf{v}_{\textrm{\tiny est}})\|}{\beta(\bdsb{\mathcal{V}})}\left\|\mathbf{F}_{\mathcal{A}}(\mathbf{v})-\mathbf{F}_{\mathcal{A}}(\mathbf{v}')\right\|.\nonumber
\end{align}

Using the norm inequality $\|\cdot\|\leq\|\cdot\|_F$ and the result we previously proved in \eqref{lipschitz_F}, we have
\begin{align}
	\left\|\mathbf{F}_{\mathcal{A}}(\mathbf{v}^k)-\mathbf{F}_{\mathcal{A}}(\mathbf{v}_{\textrm{\tiny est}})\right\|\leq
	\sqrt{(\mathbf{v}^k-\mathbf{v}_{\textrm{\tiny est}})^T\mathbf{M}(\mathbf{v}^k-\mathbf{v}_{\textrm{\tiny est}})}.
\end{align}

By denoting $\epsilon=\|\mathbf{z}_{\mathcal{A}}-\mathbf{f}_{\mathcal{A}}(\mathbf{v}_{\textrm{\tiny est}})\|$ as the noise level and $$\phi_k=  \frac{\left(\mathbf{v}^k-\mathbf{v}_{\textrm{\tiny est}}\right)^T\mathbf{M}\left(\mathbf{v}^k-\mathbf{v}_{\textrm{\tiny est}}\right)}{\left\|\mathbf{v}^k-\mathbf{v}_{\textrm{\tiny est}}\right\|^2}$$ as the Rayleigh quotient of $\mathbf{M}$ in the $k$-th iteration, we have

\begin{align}\label{new.recursion}
	\left\|\mathbf{v}^{k+1}-\mathbf{v}_{\textrm{\tiny est}}\right\|
   	&\leq \frac{1}{2}\sqrt{ \frac{\phi_k}{\beta(\bdsb{\mathcal{V}})}}\left\|\mathbf{v}^k-\mathbf{v}_{\textrm{\tiny est}}\right\|^2
	+ \frac{\epsilon\sqrt{2 \phi_k}}{\beta(\bdsb{\mathcal{V}})}\left\|\mathbf{v}^k-\mathbf{v}_{\textrm{\tiny est}}\right\|.\nonumber
\end{align}
Then, the result in Lemma 1 is readily obtained by setting $\rho_k=\left\|\mathbf{v}^k-\mathbf{v}_{\textrm{\tiny est}}\right\|$.

\section{Proof of Proposition \ref{prop_phi}}\label{proof_prop_phi}
To maintain tractability, the $\phi(\bdsb{\mathcal{V}})$ in the COP metric is approximated by the first Rayleigh quotient $\phi_0$ resulting from initialization. Therefore, we upper bound $\phi_0$ assuming that the noise in the PMU is negligible\footnote{If the noise is not small enough to be neglected, the metric becomes random and needs to be optimized on an average whose expression cannot be obtained in close-form. Thus we consider the small noise case and neglect it.}, then $\mathbf{z}_{\mathcal{V}}= \mathbf{v}_{\textrm{\tiny true}}\approx\mathbf{v}_{\textrm{\tiny est}}$ and therefore $\mathbf{v}^0-\mathbf{v}_{\textrm{\tiny est}} \approx\left(\mathbf{I}_{2N}-\mathbf{J}_{\mathcal{V}}\right)\left(\mathbf{v}_{\textrm{\tiny prior}}-\mathbf{v}_{\textrm{\tiny est}}\right)$, implying that
\begin{align*}
	\phi_0
	\approx
	 \frac{\left(\mathbf{v}_{\textrm{\tiny prior}}-\mathbf{v}_{\textrm{\tiny est}}\right)^T\left(\mathbf{I}_{2N}-\mathbf{J}_{\mathcal{V}}\right)^T\mathbf{M}\left(\mathbf{I}_{2N}-\mathbf{J}_{\mathcal{V}}\right)\left(\mathbf{v}_{\textrm{\tiny prior}}-\mathbf{v}_{\textrm{\tiny est}}\right)}{\left(\mathbf{v}_{\textrm{\tiny prior}}-\mathbf{v}_{\textrm{\tiny est}}\right)^T\left(\mathbf{I}_{2N}-\mathbf{J}_{\mathcal{V}}\right)^T\left(\mathbf{I}_{2N}-\mathbf{J}_{\mathcal{V}}\right)\left(\mathbf{v}_{\textrm{\tiny prior}}-\mathbf{v}_{\textrm{\tiny est}}\right)}.
\end{align*}
Considering the idempotence of $\left(\mathbf{I}_{2N}-\mathbf{J}_{\mathcal{V}}\right)=\left(\mathbf{I}_{2N}-\mathbf{J}_{\mathcal{V}}\right)^2$ in the numerator, the approximate bound is obtained as \eqref{omega_OPT}.

\section{Proof of Remark \ref{performance}}\label{proof_cond_perf}
Denoting $\lambda_n[\cdot]$ as the $n$-th eigenvalue of a matrix, the lower bound of the MSE is given by the trace of the FIM
\begin{align*}
	\mathbb{E}\left[\left\|\mathbf{v}_{\textrm{\tiny est}}-\mathbf{v}_{\textrm{\tiny true}}\right\|^2\right]
	&= \mathrm{Tr}\left[\mathbf{F}_{\mathcal{A}}^T(\mathbf{v}_{\textrm{\tiny true}})\mathbf{F}_{\mathcal{A}}(\mathbf{v}_{\textrm{\tiny true}})\right]^{-1}\\
	&= \sum_{n=1}^{2N}\lambda_n^{-1}\left[\mathbf{F}_{\mathcal{A}}^T(\mathbf{v}_{\textrm{\tiny true}})\mathbf{F}_{\mathcal{A}}(\mathbf{v}_{\textrm{\tiny true}})\right]\\
	&\leq 2N\lambda_{2N}^{-1}\left[\mathbf{F}_{\mathcal{A}}^T(\mathbf{v}_{\textrm{\tiny true}})\mathbf{F}_{\mathcal{A}}(\mathbf{v}_{\textrm{\tiny true}})\right],
\end{align*}
where the last inequality is obtained by bounding the inverse of each eigenvalue by the inverse of the minimum eigenvalue. From \eqref{beta}, $\beta(\bdsb{\mathcal{V}})$ equivalently serves as a lower bound of the minimum eigenvalue of the FIM
\begin{align}
	\beta(\bdsb{\mathcal{V}})
	\leq \lambda_{\min}\left[\mathbf{F}_{\mathcal{A}}^T(\mathbf{v}_{\textrm{\tiny true}})\mathbf{F}_{\mathcal{A}}(\mathbf{v}_{\textrm{\tiny true}})\right]
\end{align}
and therefore $\lambda_{2N}^{-1}\left[\mathbf{F}_{\mathcal{A}}^T(\mathbf{v}_{\textrm{\tiny true}})\mathbf{F}_{\mathcal{A}}(\mathbf{v}_{\textrm{\tiny true}})\right]\leq 1/\beta(\bdsb{\mathcal{V}})$. Thus, the result in the condition follows.

\bibliographystyle{IEEEtran}
\bibliography{../ref_general,../ref_opt_PMU_placement,../ref_power_system_SE,../ref_dist_opt}


\end{document}